\documentclass[sort&compress,3p]{elsarticle}

\input{CRZ14_TTIE_preamble}

\begin{document}
\title{A Tensor-Train accelerated solver for integral equations\\ in complex geometries}
\author[umich]{Eduardo Corona} \ead{coronae@umich.edu}
\author[nyu]{Abtin Rahimian\corref{cor}} \ead{arahimian@acm.org}
\author[nyu]{Denis Zorin} \ead{dzorin@cs.nyu.edu}
\address[umich]{Department of Mathematics, University of Michigan,  Ann Arbor, MI 48109}
\address[nyu]{Courant Institute of Mathematical Sciences, New York University, New York, NY 10003}
\cortext[cor]{Corresponding author}
\begin{abstract}
  We present a framework using the Quantized Tensor Train (\QTT)
  decomposition to accurately and efficiently solve volume and
  boundary integral equations in three dimensions.
  We describe how the \QTT\ decomposition can be used as a hierarchical
  compression and inversion scheme for matrices arising from
  the discretization of integral equations.
  For a broad range of problems, computational and storage costs of
  the inversion scheme are extremely modest $\O{\log N}$ and once the
  inverse is computed, it can be applied in $\O{N \log N}$.

  We analyze the \QTT\ ranks for hierarchically low rank matrices and
  discuss its relationship to commonly used hierarchical compression
  techniques such as \FMM\ and \HSS.
  We prove that the \QTT\ ranks are bounded for translation-invariant
  systems and argue that this behavior extends to non-translation
  invariant volume and boundary integrals.

  For volume integrals, the \QTT\ decomposition provides an efficient
  direct solver requiring significantly less memory compared to other
  fast direct solvers.
  We present results demonstrating the remarkable performance of the
  \QTT-based solver when applied to both translation and
  non-translation invariant volume integrals in \threed.

  For boundary integral equations, we demonstrate that using a
  \QTT\ decomposition to construct preconditioners for a Krylov
  subspace method leads to an efficient and robust solver with a small
  memory footprint.
  We test the \QTT\ preconditioners in the iterative solution of an
  exterior elliptic boundary value problem (Laplace) formulated as a
  boundary integral equation in complex, multiply connected
  geometries.

\end{abstract}

\maketitle

\section{Introduction\label{sec:intro}}
We aim to efficiently and accurately solve equations of the form

\begin{equation}
  a \sigma(x) + \int_{\Omega} b(x) K(x,y) c(y) \sigma(y) \d \Omega_y =
  f(x),\qquad \text{for all } x \in \Omega, \tag{IE}
  \label{eq:basic-integral}
\end{equation}
where $\Omega$ is a domain in $\mathbb{R}^{3}$ (either a boundary or a
volume).
When $a \ne 0$, the integral equation is Fredholm of the second kind,
which is the case for all equations presented in this work.
A large class of physics problems formulated as {\PDE}s may be cast in
this equivalent form. $K(x,y)$ for these type of problems is a
\textit{kernel function} derived from the fundamental solution of the
\PDE.
The advantages of integral equation formulations include reducing
dimension of the problem from three to two and improved conditioning
of the discretization.

The kernel function $K(x,y)$ is typically singular as $x$ approaches
$y$ but is smooth otherwise.
For the purposes of this paper, we also assume it is not highly
oscillatory.

A discretization of \pr{eq:basic-integral} produces a linear system
of equations
\begin{equation}
  \la{A\sigma = f},\tag{LS}
  \label{eq:Asigma}
\end{equation}
where $\la{A}$ is a dense $N \times N$ matrix.
Krylov subspace methods such as \GMRES~\cite{GMRES} coupled with the
rapid evaluation algorithms such as \FMM~\cite{greengard1987fast} are
widely used to solve this system of equations.
However, the performance of the iterative solver is directly affected
by the eigenspectrum of \pr{eq:basic-integral}.

The eigenspectrum of the system, while typically independent of the
resolution of the discretization, can vary greatly, depending on the
geometric complexity of $\Omega$ and the kernel $K$ in particular.
When the spectrum is clustered away from zero, the system is solved in
a few iteration using a suitable iterative method. However, this is
not the case for a number of problems of interest (e.g., when
different parts of the boundary approach each other).
Such problems may either be solved by constructing an effective
preconditioner for the iterative solver or by using \emph{direct
  solvers}, in which the system is solved in a fixed time independent
of the distribution of the spectrum.

There have been a number of efforts to develop robust, fast direct
solvers with linear complexity for systems given in \pr{eq:Asigma}.
When $\Omega$ is a contour in the plane, extremely efficient $\O{N}$
algorithms such as \cite{MR2005} exist.
These algorithms may be extended to volumes in \twod\ and surfaces in
\threed, producing direct solvers with complexity $\O{N^{3/2}}$
\cite{GYMR2012, ho2012fast, G2011thesis}.
More recently, approaches that aim to extend linear complexity to
Hierarchically Semi-Separable (\HSS) matrices have been developed
\cite{CMZ14, ho2015hierarchical_ie}.
Furthermore, a general inverse algorithm has been proposed for \FMM\
matrices \cite{ambikasaran2014inverse}.

For \twod\ problems, these types of methods have excellent performance and
remain practical even at high target accuracies, e.g., \sci{-10}.
However, for volume or boundary integral equations in \threed,
especially in complex geometries, algorithmic constants in the
complexity of these methods grow considerably as a function of
accuracy.
In particular, the compressed form of the inverse typically requires a
very large amount of storage per degree of freedom, limiting the range
of practical target accuracies or problem sizes that one can solve.
This also precludes  efficient parallel implementation, due to the need to
store and communicate large amounts of data.

Basic iterative solvers (requiring only matrix-vector multiplication)
and direct solvers, represent two extremes of the spectrum of
\emph{preconditioned iterative solvers}, as a direct solver can be
viewed as a preconditioned solver with a perfect preconditioner
requiring one iteration to converge.
These also represent two extremes in the tradeoff between memory and
time spent for computing the preconditioner as well as the cost of
each solve.
A direct solver requires a lot of memory and precomputation time
for  the inverse matrix, but solving a system for a specific
right-hand side is typically very fast.
On the other extreme, a non-preconditioned iterative solver requires
only a matrix-vector multiplication function, either requiring no
precomputation, or compression of the matrix (but not computing its
inverse).
However, each solve in this case may require a large number of
iterations.

Varying the accuracy of the approximate inverse preconditioner leads to
intermediate solvers, with reduced storage and time required for
precomputation, but increased time needed for solves.
By varying this accuracy, we can find a ``sweet spot'' for a
particular type of problem and let the practitioner strike a
reasonable trade-off between precomputation time and solve time,
within the available memory budget.

\subsection{Contributions and outline}
We present an effective and memory-efficient preconditioned solver for \pr{eq:Asigma} based on the quantized \TTrain\ decomposition (\QTT) \cite{oseledets2010tt,khoromskij2009dlog}. The \QTT\ decomposition allows for the compact representation and fast arithmetic of structured matrices by recasting them in tensor form. 

In this work, we frame this process in the context of hierarchical compression and inversion for matrix $\la{A}$. We show that for a range of target accuracies, \QTT\ decomposition achieves significantly higher compression by finding a common basis for all interactions at a given level of the hierarchy. We argue this makes it a natural fit for the solution of integral equations, and discuss how, for certain problems, it can achieve superior performance versus commonly used hierarchical compression techniques such as \FMM\ and variants of $\Hmat$-matrices.

We prove rank bounds for the \QTT\ ranks of $\la{A}$ for translation-invariant systems, as well as for a family of non-translation invariant systems obtained from volume integral equations in \pr{sec:integraleq}, and provide evidence that this behavior extends to boundary integrals in complex geometries. In our experiments in \pr{sec:results}, we find that the inverse $\la{A}^{-1}$ is also highly compressible, displaying comparable rank behavior to that of $\laT{A}$ in all cases considered. 

In our presentation of the \QTT\ inversion process in \pr{sec:inversion}, we contribute novel strategies to precondition
the \QTT\ inversion algorithms, representing the inverse as a product of matrix factors in the \QTT\ form to achieve faster computation. 

Finally, we perform an extensive series of numerical experiments to evaluate the performance and experimental scaling of the \QTT\ approach for volume and boundary integral equations of interest. In both instances, we confirm key features of \QTT\ that distinguish it from existing fast direct solver techniques:

\begin{itemize}
\item \lh{Logarithmic scaling.} Given a target accuracy, the matrix
  compression and inversion steps based on rank-revealing techniques
  as well as the storage are observed to scale no faster than
  $\O{\log~N}$.
  This \emph{sublinear} scaling, remarkably, implies that the relative
  cost of the computation stage prior to a solve (direct or
  preconditioned iterative) becomes negligible as $N$ grows.
\item \lh{Memory efficiency.} One of the main issues of current fast
  direct solvers is that even when they retain linear scaling, they
  require significant storage per degree of freedom for problems in
  \threed.
  Due to its logarithmic scaling, the \QTT\ decomposition completely
  sidesteps this, requiring no more than $100$~\abbrev{MB} of memory
  for the compression and inversion of systems with a large number of
  unknowns ($N \sim 10^7{-}10^8$).
\item \lh{Fast matrix-vector and matrix-matrix operations.} If both of
  the operands are represented in the \QTT\ format, $\O{\log N}$
  matrix-vector and matrix-matrix apply algorithms are available.
  Furthermore, $\O{N \log N}$ matrix-vector apply algorithm
  exists for the multiplication of a matrix in the \QTT\ format with a
  non-\QTT-compressed vector.
\end{itemize}

In \pr{ssc:tt-vie}, we present results demonstrating the remarkable performance of a \QTT-based \emph{direct} solver for translation and non-translation invariant volume integral equations in \threed\  for target accuracies up to \sci{-10} (e.g., arising in the context of Poisson--Boltzmann or Lippmann--Schwinger equations), for which existing direct solvers are generally not practical.

In \pr{ssc:bie-lattice}, we demonstrate that using the \QTT\ framework to construct preconditioners for a Krylov subspace method leads to an efficient and robust solver with a small memory footprint for boundary integral equations in complex geometries. We test this \QTT-based preconditioner in the iterative solution of an exterior elliptic boundary value problem (Laplace) formulated as a boundary integral equation in complex, multiply connected geometries\Mdash/a model problem for a range of problems with low frequency kernels, e.g., particulate Stokes flow, electrostatics and magnetostatics. We show that the \QTT-based approach significantly outperforms other state-of-the-art methods in memory efficiency, while being comparable in speed.

\subsection{Scope and limitations}
In this work we only consider linear systems from Fredholm integral
equation of the second kind, uniformly refined octree partitions of
the domain (i.e., no adaptivity), and serial versions of the algorithms
only.
Extension of some of the \QTT\ compression and inversion
algorithms to obtain good parallel scaling is an interesting research
direction in its own right.

The main limitation of the current framework as a hierarchical
inversion tool of linear operators is the quartic\footnote{Assuming
  the inverse has similar \QTT\ ranks to that of $\la{A}$, which is
  the case for systems arising from integral equations.} dependence of
performance on the \QTT\ ranks in the inversion algorithm, see
\pr{ssc:inv-complexity}. To a large extent this fact necessitates the
use of \QTT\ to obtain a preconditioner (rather than a direct solver)
for problems with high variation of coefficients.

\subsection{Related work}
Solution of \pr{eq:Asigma} is computed with low computational
complexity either iteratively or directly.
The former leverages rapid evaluation algorithms such as
\FMM\ combined with Krylov subspace methods and the latter is based on
fast direct solvers.
At the heart of rapid evaluation or fast direct inversion algorithms
lies the observation that, due to the properties of the underlying
kernel, off-diagonal matrix blocks have low numerical rank.
Using a hierarchical division of the integration domain
$\Omega$\Mdash/represented by a tree data structure\Mdash/these
algorithms exploit this low rank property in a multi-level fashion.

\subsubsection{Iterative solvers for integral equations}
During the 1980s, the development of rapid evaluation algorithms for
particle simulations such as the Fast Multipole
Method~\cite{rokhlin1985, greengard1987fast, rokhlin1997},
Panel~Clustering~\cite{PanelClustering}, and
Barnes-Hut~\cite{BarnesHut} as well as the development of Krylov
subspace methods for general matrices such as \GMRES~\cite{GMRES} or
\abbrev{Bi-CGSTAB}~\cite{BiCGSTAB} provided $\O{kN}$ or $\O{kN\log N}$
frameworks for solving systems of the form \pr{eq:Asigma}, where $k$
is the required number of iterations in the iterative solver and
directly depend on the conditioning of the problem.

\subsubsection{Direct solvers for integral equations}
Fast direct solvers avoid conditioning issues, and lead to substantial
speedups in situations where the same equation is solved for multiple
right-hand-sides.  Fast direct solvers for \HSS\ matrices were
introduced in \cite{starr_rokhlin, greengard1996direct}.
\citet{MR2005} describe an optimal-complexity direct solver for
boundary integrals in the plane.
Extensions of this solver to surfaces in \threed\ were developed in
\cite{GYMR2012, ho2012fast, G2011thesis} and in the related works
of \cite{Gu06sparse, Gu06ULV, xia2010}.
As we mentioned earlier, for \threed\ surfaces, the complexity of inversion
in these algorithms is $\O{N^{3/2}}$.
Since this increase is due to the growth in rank of off-diagonal
interactions, additional compression is required to regain optimal
complexity.
\citet{CMZ14} achieved this for volume integral equations
in \twod\ by an additional level of hierarchical compression of the blocks
in the \HSS\ structure.
While a similar approach can be applied to surfaces in \threed, it results
in a significant increase in required memory for the inverse storage
as well as longer computation times.

\citet{ho2015hierarchical_ie} proposed an alternate approach, the
Hierarchical Interpolative Factorization (\abbrev{HIF}), using
additional skeletonization levels and implemented it for both
\twod\ volume and \threed\ boundary integral equations using standard
direct solvers for sparse matrices in an augmented system. While the
structure of the algorithms suggests linear scaling, for
\threed\ problems the observed behavior is still above linear.

In a series of papers on $\Hmat$- and $\Hmat^2$-matrices, Hackbusch
and coworkers constructed direct solvers for \FMM-type matrices.
The reader is referred to \cite{borm2003hierarchical,
  2008_bebendorf_book, 2010_borm_book} for in depth surveys.
The observed complexity for integral equation operators, as reported in
\cite[Chapter 10]{2010_borm_book}, is $\O{N\log^4 N}$ for matrix
compression, $\O{N\log^3 N}$ for inversion, and $\O{N\log^2 N}$ for
solve time and memory usage, with relatively large constants.

More recently, a promising inverse \FMM\ algorithm was introduced
\cite{ambikasaran2014inverse,IFMM2015prec}, demonstrating efficient
performance and $\O{N}$ scaling for \twod\ and \threed\ volume computations.

\subsubsection{Preconditioning techniques for integral equations}
The convergence rate of \GMRES\ is mainly controlled by the distribution
of the eigenvalues in the complex plane
\cite{Nachtigal1992a,Benzi2002}. Preconditioning techniques aim to
improve the rate by clustering the eigenvalues away from zero.
For excellent reviews on the general preconditioning techniques the
reader is referred to \cite{Benzi2002,Wathen2015}.
Here, we focus on the preconditioners tailored for linear systems
arising from the discretization of integral equations.

Most preconditioning techniques for integral equations can be
categorized as sparse approximate inverse (\abbrev{SPAI}) or multi-level
schemes.
\abbrev{SPAI} seeks to find a preconditioner $\la{M}$ that satisfies
$\min \|\la{I-M\tilde{A}}\|_F$ subject to some
constraint on the sparsity pattern of $\la{M}$\Mdash/typically chosen
\emph{a priori}.
Here $\la{\tilde{A}}$ denotes an approximation to $\la{A}$ to make the
optimization process economical.
Multi-level preconditioning methods include stationary iteration
techniques like multigrid and single-grid low-accuracy inverse.

Apart from \abbrev{SPAI} and multi-level methods, some authors used
incomplete factorizations as preconditioner for integral equations
\cite[and references therein]{Wang2007}.
However, because of their potential instabilities, difficulty of
parallelization, lack of algorithmic scalability, and non-monotonic
performance as a function of fill-ins \cite{Benzi2002} they are less
popular for integral equations.

\para{Sparse approximate inverse preconditioners (\abbrev{SPAI})}
\abbrev{SPAI}s with sparsity and approximation based on geometric
adjacency (e.g. \FMM\ tree) are a popular choice for boundary integral
equations \cite{Vavasis1992, Nabors1994, Tausch1996, Grama1996a,
  Tausch1997, Chan1997, Carpentieri2003a, Carpentieri2005a, Wang2007},
due to their low computation and application cost and scalability.
\citet{Vavasis1992} introduced the (\emph{mesh neighbor} scheme), with
the sparsity pattern defined for an \FMM\ octree/quadtree cell by
near-interaction cells, and \emph{hierarchical clustering} improving
the mesh-neighbor scheme using first-order multipoles from far boxes.
Variations of these schemes are found in \cite{Nabors1994,Tausch1996,
  Pissoort2006}.
For certain problems, mesh-neighbor is effective in reducing the
number of iterations but its performance depends on the grid size, and
it is most effective when the far interactions are negligible,
(cf. \cite{Pissoort2006}). In general the effectiveness of \abbrev{SPAI}
preconditioners with sparsity pattern based on \FMM\ adjacency  deteriorates
with increased tree depth \cite{Carpentieri2003a, Carpentieri2005a}.
\citet{Tausch1997} incorporated the far field by including a
first-order multipole expansion, which required solving a system of
size \O{\log N} for each set of target points in a box.
The resulting preconditioner is not sparse but has constant blocks for
far boxes and can be applied efficiently.

\citet{Carpentieri2003a} and \citet{Giraud2007} observed that the
\abbrev{SPAI} preconditioners are effective in clustering most of the
eigenvalues but leave a few close to the origin and removing them
needs problem-dependent parameter tuning.
To remedy, these authors proposed low-rank updates to the
preconditioner using the eigenvectors corresponding the smallest
eigenvalues of $\la{MA}$.

\para{Multi-level methods} These schemes were introduced to address
the shortcomings of \abbrev{SPAI}.
\citet{Grama1996a} proposed a low-order and low-accuracy iterative
inner solver as a multi-level preconditioner, which was very effective
in reducing the number of iterations.
However, the ill-conditioning of the system caused a high number of
inner iterations and consequently the scheme was not time effective.
\citet{Carpentieri2005a} pursued this direction further and used a
mesh-neighbor \abbrev{SPAI} as the preconditioner for the inner solver.
\citet{Gurel2010} used a similar approach for solving electromagnetic
scattering problems.

Authors have opted for \abbrev{SPAI} or iterative multi-level methods
mainly because these methods have \O{N} complexity in time and memory
by construction.
Recently, leveraging randomized algorithms and fast direct solver
schemes, preconditioners with competitive complexity and much better
eigenspectrum clustering have been proposed
\cite{Bebendorf2005,Quaife2013a,Ying2014,IFMM2015prec}.

\citet{Bebendorf2005} and \citet{Benedetti2008} constructed
compression schemes for boundary integral equations based on
\Hmat-matrix approximation.
To solve the resulting system, a low accuracy \Hmat-LU with accuracy
$\acc_p$ was used as preconditioner for the iterative
solver.
The complexity of \Hmat-LU is $\bigO(|\log\acc_p|^4 N \log^2 N)$.
In \cite{Benedetti2008}, the best speedup was achieved when using
preconditioner with $\acc_p=\sci{-1}$ and higher accuracies did not
improve the time due to preconditioner setup and apply overhead.

\citet{Quaife2013a} proposed \FMM- and multigrid-based preconditioners
for the second-kind formulation of the Laplace equation in \twod.
They demonstrated that even with the exact inversion in constructing
the mesh-neighbor preconditioner, \GMRES\ still requires many
iterations.
To construct a better preconditioner, another level of neighbors were
included and inverted using \abbrev{ILU}(\sci{-3}) combined with
Sherman--Morrison--Woodbury formula.
The preconditioner was further improved by including a rank \O{\log N}
approximation of the residual matrix $\la{A-A_\lbl{near}}$, where
$\la{A_\lbl{near}}$ denotes the sparsified matrix (approximately)
inverted to construct the preconditioner.

\citet{Ying2014} constructed a very effective preconditioner for the
iterative solution of integral equation formulation of the
Lippmann--Schwinger equation.
The asymptotic setup and application time of the preconditioner as
well as its memory requirements are similar to those of direct solvers
but the preconditioner has a smaller size.

\citet{IFMM2015prec} presented \abbrev{IFMM} as a fast direct solver,
where the matrix is converted to an extended sparse matrix and its
sparse inverse is constructed by careful compression and redirection
of the fill-in blocks resulting in \O{N} complexity for the algorithm.
To achieve high-accuracy solutions in a cost-effective way, the
authors proposed using a low-accuracy \abbrev{IFMM} as a
preconditioner in \GMRES.

\subsubsection{Low-rank tensor approximation of linear operators}
Tensor factorizations were originally designed to tackle
high-dimensional problems in areas of physics such as quantum
mechanics.
To be able to perform computations for these problems, the curse of
dimensionality
has to be overcome.
The \emph{\TTrain\ decomposition} is one of the tensor
representation methods developed for this purpose.
Other factorization methods include the \abbrev{CP}
(\abbrev{CANDECOMP/PARAFAC}), Tucker, and Hierarchical Tucker
\cite{hackbusch2005hierarchical,kolda2009survey} decompositions; more
details can be found in \cite{grasedyck2013literature,
  grasedycknotes,khoromskij2015survey}.

It was observed that schemes of this type can be useful for
low-dimensional problems, recast in the tensor form.
\emph{Quantized-\TT}~(\QTT) algorithms reshape vectors or matrices
into higher dimensional tensors (i.e. \emph{tensorize or quantize})
and then compute a \TTrain~(\TT) decomposition with low tensor
rank.
Approximation of $2^d \times 2^d$ matrices as $d$ dimensional tensors
was first observed in
\cite{oseledets2009tensors,oseledets2010approximation}. The quantized
tensor train approximation was first proposed and analyzed for a
family of function-related vectors in
\cite{khoromskij2009dlog,khoromskij2011dlog}, including discretized
polynomials, which were shown to have exact low-rank representations.

The observation that certain kinds of structured matrices may be
efficiently represented using the \QTT\ format has been made for
Toeplitz matrices \cite{olshevsky2006tensor, oseledets2011tensor}, the
Laplace differential operator and its inverse \cite{kazeev2012low,
  oseledets2010approximation}, general {\PDE}s and eigenvalue problems
\cite{khoromskij2011dlog}, convolution operators
\cite{hackbusch2011convolution}, and the
\FFT~\cite{dolgov2012fft}.
Recent developments feature its use to solve multi-dimensional
integro-differential equations arising in fields such as quantum
chemistry, electrostatics, stochastic modeling and molecular dynamics
\cite{khoromskij2015survey}.
In the context of boundary integral equations, it has additionally
been used to speed up the quadrature evaluation for \abbrev{BEM}
\cite{khoromskij2001bem}.
We note that in the context of volume integral equations, in
\cite{khoromskaia2014Newton}, low rank Canonical
decomposition~(\abbrev{CP}) and Tucker decomposition representations
were obtained for the tensorization of the Newton kernel, as well as
for a related class of translation invariant kernels. Hybrid formats
with $\Hmat$ matrices, such as the blended kernel approximation
\cite{hackbusch2002blended} and Hierarchical Tucker (\abbrev{HTK})
\cite{hackbusch2005hierarchical,austin2015parallel} have also been
used to approximate tensorized volume integral kernels, with
$\O{N^{1/D}\log N}$ storage requirements in $D$ dimensions.

Given a linear system whose corresponding matrix can be efficiently
represented with \QTT, there exist several algorithms to compute a
\QTT\ representation of its inverse.
In this work, we use the alternating minimal energy (\abbrev{AMEN})
and the density matrix renormalization group (\abbrev{DMRG}) methods
as proposed in \cite{oseledets2012solution, dolgov2013amen1,
  dolgov2013amen2}.
Another such method is the Newton--Hotelling--Schulz algorithm
\cite{hackbusch2008newton, olshevsky2008superfast}.

\section{Background: Quantized Tensor-Train decomposition \label{sec:background}}
In this section, we first review the general \TTrain\ (\TT)
decomposition, and briefly discuss the properties and computational
complexity of state-of-the-art tensor compression algorithms.
We then review the \emph{quantized}~\TTrain~(\QTT) used to compress
\emph{tensorized} vectors consisting of samples of functions on a
hierarchically subdivided domain.
Matrices arising from the discretization of \pr{eq:basic-integral} in
this setting are interpreted as tensorized operators acting on such
tensorized vectors.

\subsection{Nomenclature}
We use different typefaces to distinguish between different
mathematical objects, namely we use:
\begin{compactitem}[$\qquad-$]\vspace{6pt}
\item Roman letters for continuous functions: $f(x)$, $K(x,y)$;
\item calligraphic letters for multidimensional arrays and tensors:
  $\cA[f](i_1,\dots,i_d)$, $\cA[K](i_1,j_1,\dots,i_d, j_d)$;
\item sans-serif for vectors and matrices: $\laT{f}(i)$,
  $\laT{K}(i,j)$; and
\item typewriter for the \TT\ decompositions of tensors:
  $\mathtt{f}$, $\mathtt{K}$.
\end{compactitem}\vspace{6pt}
We use Matlab's notation for general array indexing and reshaping.
Given a multi-index $(i_1,\ldots,i_d)$, we will denote the
corresponding one-dimensional index obtained by ordering multi-indices
lexicographically, by placing a bar on top and removing commas between
indices: $i =\overline{i_1i_2 \cdots i_d}$.
This mapping from multi-indices to one-dimensional index defines a
conversion of a multidimensional array to a vector which we denote
$\la{b} =\vec(\cA[b])$, with $\la{b}(\overline{i_1i_2 \cdots i_d}) =
\cA[b](i_1,i_2, \ldots, i_d)$.

\subsection{Tensor train decomposition}
The \TTrain\ decomposition is a highly effective representation for
compact low-rank approximation of tensors \cite{oseledets2010tt}.

\begin{definition}
  Let $\cA[A]$ be a $d$-dimensional tensor, sampled at $N =\prod_{k=1}^d
  n_k$ points, indexed by $(i_1,i_2,\ldots,i_d)$, $i_k\le n_k$. The
  \TT\ decomposition $\ttA[A]$ of the tensor $\cA[A]$ is given by
  \begin{equation}
    \ttA[A](i_1,i_2,\ldots,i_d) \defeq
    \sum_{\alpha_1,\ldots,\alpha_{d-1}} {\tensor{G}_1(i_1,\alpha_1)
      \tensor{G}_2(\alpha_1,i_2,\alpha_2) \dots
      \tensor{G}_d(\alpha_{d-1},i_d)},
    \label{eq:TTdeccores}
  \end{equation}
  where, each two- or three-dimensional $\tensor{G}_k$ is called the
  \ordinal{k} tensor core. Auxiliary indices have the range $\alpha_k=1,
  \ldots, r_k$, where $r_k$ is called the \ordinal{k} \TT-rank.
\end{definition}

For algorithmic purposes, it is often useful to introduce dummy
indices $\alpha_0$ and $\alpha_d$, and let the corresponding
\TT\ ranks $r_0=r_d=1$; in this way, we can view all cores as
three-dimensional tensors.
We note that the \TT-ranks determine the number of terms in the
decomposition.
A \TT\ approximation of a $\tensor{A}$ is a tensor $\ttA[A]$ in
\TT\ format, with minimal \TT-ranks such that $||\ttA[A] -
\tensor{A}||_F < \acc$, where $\acc$ is a given accuracy.

The key property of the \TT\ decomposition is that a nearly-optimal
approximation of a matrix can be computed efficiently, i.e., This
approximation problem is linked to the low rank approximation of the
tensor's \emph{unfolding matrices}.

\begin{definition}
  For a tensor $\cA[A]$ of dimension $d$, the \ordinal{k}
  \emph{unfolding matrix} $\laT{A}^k$ is defined entrywise as
  \begin{equation}
    \laT{A}^{k}(p_k,q_k) = \laT{A}^{k}(\overline{i_1i_2\cdots i_k} ,
    \overline{i_{k+1}\cdots i_d}) = \cA[A](i_1,i_2,\cdots,i_d)
    \quad\text{for}\quad k=1,\ldots, d,
  \end{equation}
  where $p_k = \overline{i_1\cdots i_k}$ and $q_k =
  \overline{i_{k+1}\cdots i_d}$ are two flattened indices.
  Using Matlab's notation,
  \begin{equation}
    \laT{A}^{k} = \reshape\left(\cA[A],\prod_{\ell=1}^k
        {n_\ell},\prod_{\ell=k+1}^d {n_\ell}\right).
  \end{equation}
\end{definition}

By contracting the first $k$ and the last $d-k$ cores of a
\TT\ decomposition $\ttA[A]$, we observe that the corresponding
\ordinal{k} unfolding matrix is of matrix rank $r_k$.
The low \TT-rank approximation problem of a tensor $\tensor{A}$ is
thus linked to low-rank approximations of its unfolding matrices
$\laT{A}^k$.

\subsection{\TT\ approximation algorithms}
A low-rank \TT\ decomposition can be obtained by a sequence of
low-rank approximations to $\laT{A}^{k}$ (e.g., using a sequence of
truncated \abbrev{SVD}s). More generally, given a low-rank matrix
approximation routine, a generic algorithm to compute the
\TT\ decomposition proceeds as in \pr{alg:TT-dec}.
\begin{algorithm}[!b]
  \centering
  \begin{algorithmic}[1]
    \setlength\commLen{.5\linewidth} 
    \REQUIRE Tensor $\tensor{A}$ ($d$-dimensional), and target accuracy $\acc$
    \STATE $\laT{M}_1 = \laT{A}^{1}$ \COMMENT{First unfolding matrix}
    \STATE $r_0 = 1$
    \STATE $\acc_\lbl{LR} = \acc ||A||_F / \sqrt{d-1}$
    \FOR{$k=1$ to $d-1$}
    \STATE $[\laT{U}_k$,$\laT{V}_k] = \mathtt{lowrank\_approximation}(\laT{M}_k,\acc_\lbl{LR})$
    \STATE $r_k =\mathtt{size}(\laT{U}_k,2)$ \COMMENT{\ordinal{k} \TT\ rank}
    \STATE $\tensor{G}_k = \mathtt{reshape}\big( \laT{U}_k , [r_{k-1},n_k,r_k]\big)$ \COMMENT{\ordinal{k} \TT\ core}
    \STATE $\laT{M}_{k+1} = \mathtt{reshape}\Big( \laT{V}_k , [r_kn_{k+1},\prod_{\ell=k+2}^{d} {n_{\ell}}]\Big)$ \COMMENT{$\laT{M}_{k+1}$ corresponds to the \ordinal{(k+1)} unfolding matrix of $\tensor{A}$}
    \ENDFOR
    \STATE $\tensor{G}_d = \mathtt{reshape}\big( \laT{M}_d, [r_{d-1},n_d,1]\big)$ \COMMENT{Set last core to the right factor in the low rank decomposition}
    \RETURN $\ttA[A]$
  \end{algorithmic}
  \algcaption{alg:TT-dec}{Compute a \TT\ decomposition}{}
\end{algorithm}

Using the notation given in \pr{alg:TT-dec}, a low-rank decomposition
of the unfolding matrix $\laT{A}^{k} \approx \laT{U}^k \laT{V}^k$ may
be obtained by multiplying $\laT{U}_k$ by the already computed first
$k-1$ cores $\tensor{G}_k$ and setting $\laT{V}^k = \laT{V}_k$.
\citet{oseledets2010tt} show that in \abbrev{SVD}-based
\TT\ compression, for any tensor $\tensor{A}$, when the low-rank
decomposition error $\acc$ for the unfolding matrices is optimal for
rank $r_k$
\begin{equation}
  \acc_k = \norm[F]{\laT{A}^{k}-\laT{U}^k\laT{V}^k} =
  \min_{\rank(\laT{B}) \leq r_k} {\norm[F]{\laT{A}^{k} - \laT{B}}} \qquad (k=1,\dots,d-1),
\end{equation}
the corresponding \TT\ approximation $\mathtt{A}$ satisfies
\begin{equation}
  \norm[F]{\tensor{A} - \mathtt{A}}^2 \leq \sum_{k=1}^{d-1} {\acc_k^2}.
\end{equation}

Given prescribed upper bounds $r_k$ for the \TT\ ranks, there exists a
unique Frobenius-norm optimal approximation in the \TT\ format
$\mathtt{A}_\lbl{optimal}$ and the approximation $\mathtt{A}$ obtained
by the \abbrev{SVD}-based \TT\ algorithm is quasi-optimal
\begin{equation}
  \norm[F]{\tensor{A}-\mathtt{A}} \leq
  \sqrt{d-1}\norm[F]{\tensor{A}-\mathtt{A}_\lbl{optimal}}.
\end{equation}

The direct application of \pr{alg:TT-dec}, where the ranks $r_k$ are
obtained using a rank-revealing decomposition still leads to
relatively high computational cost, $\O{N}$ or higher, which is
exponential in the dimension $d$.
Fortunately, \TT\ approximation algorithms with much better scaling
are available. Throughout this work, we employ the multi-pass
Alternating Minimal Energy (\abbrev{AMEN}) Cross algorithm, based on
\cite{oseledets2010tt}, available in the
TT-Toolbox~\cite{toolboxtt}. For tensors with bounded maximum
\TT\ ranks, this algorithm scales \emph{linearly} with dimension
$d$. In the \abbrev{AMEN} Cross algorithm, a low-\TT-rank
approximation is initially computed with fixed \TT-ranks and is
improved upon by a series of passes through all \TT\ cores.
This algorithm is thus iterative in nature, increasing the maximum
\TT\ rank after each pass until convergence is reached.
The analysis and experiments in \cite{dolgov2013amen1,dolgov2013amen2}
show that these iterations exhibit monotonic, linear convergence to an
approximation of the original tensor for a given target accuracy
$\acc$.

\para{Quantized-\TT\ rank and mode size implications} While the
algorithms and the analysis for the \TT\ decomposition are
generic, our work focuses on their
application to function and kernel sampled in two or three dimensions
by casting them as higher dimensional tensors. This type of
\TT\ decomposition is usually referred to as  \emph{Quantized}-\TT\ or \QTT.
In the process of \emph{tensorization}, \QTT\ splits each dimension
until each tensor mode $n_k\,(k=1,\ldots,d)$ is very small in size.
For instance, a one-dimensional vector of size $N=2^d$ is converted to
a $d$-dimensional tensor with each mode of size 2 (implying that $d
\approx \log N$).

\para{Computational Complexity and Memory Requirements}
Because \abbrev{AMEN} Cross and related \QTT\ rank revealing approaches proceed by
enriching low \QTT-rank approximations, all computations are performed
on matrices of size $r_{k-1}n_k \times r_k$ or less.
Performing an \abbrev{SVD} on such matrices is $\O{r_{k-1}^2n_k^2r_k + r_k^3}$
\cite{golubvanloan}.
Other low-rank approximations such as the \abbrev{ID} \cite{ID2005compression}
have similar complexity.
As a consequence, computational complexity for this algorithm is
bounded by $\O{r^3d}$ or equivalently $\O{r^3 \log N}$, where $r =
\max(r_k)$ is the maximal \QTT-rank that may be a function of sample
size $N$ and accuracy $\acc$.

In some cases, as we will discuss in more detail below, the maximal
\QTT-rank $r$ can be bounded as a function of $N$.
For differential and integral operators with non-oscillatory kernels,
as well as their inverses, $r$ typically stays constant or grows
logarithmically with $N$ \cite{khoromskij2009dlog,kazeev2012low,
  oseledets2010approximation, oseledets2011tensor}.
If this is the case, the overall complexity of computations is
\emph{sublinear} in $N$.

\subsection{Applying the \QTT\ decomposition to function samples \label{ssc:TTfunction}}
Our goal is to use the \QTT\ decomposition to compress matrices in
\pr{eq:Asigma}.
In this section, we cast \QTT\ as an algorithm operating on a
hierarchical partition of the data to provide a better
understanding of its performance.
In the next section, we formally prove that such decompositions indeed
have low \QTT-ranks.

Let $f:\Omega \rightarrow \mathbb{R}$ be a function on $\Omega$, a
compact subset of $\mathbb{R}^D$.
We consider hierarchical partitions of $\Omega$ into disjoint subsets;
at each level of the partition, each subset is split into the same
number of subsets $n$.
This partition can be viewed as a tree $\tree{T}$ with subdomains at
different levels as nodes.
We number levels from $2$ to $d$, where $2$ is assigned to the
\emph{finest} level and $d$ to the tree root.
Hence, the domains at the finest level can be indexed with a
multi-index $(i_2, \ldots i_d)$ where $i_\ell$ indicates which of $n$
branches was taken at level $\ell$.
For each leaf domain we pick $n_1$ sample points
$x_{i_1,i_2\ldots,i_d}$, adding an additional index $i_1\le n_1$ to
the multi-index.
Thus, we define a tensor
\begin{equation}
  \cA[f]_{\tree{T}}(i_1,i_2,\ldots,i_d) = f(x_{i_1,i_2,\ldots,i_d}),
\end{equation}
with each index corresponding to a level of the tree.
Flattening this tensor yields a vector of samples
$\la{f}=\vec(\cA[f]_{\tree{T}})$.
If we compute the \TT\ decomposition given in \pr{eq:TTdeccores} for
this $d$-tensor, we obtain an approximation to $f$ as a sum of the
terms of the form
\begin{equation}
  \tensor{G}_1(\alpha_0, i_1,\alpha_1)
  \tensor{G}_2(\alpha_1,i_2,\alpha_2) \dots
  \tensor{G}_d(\alpha_{d-1},i_d,\alpha_d),
\end{equation}
where each core $\tensor{G}_\ell$ corresponds to a level of the
hierarchy in $\tree{T}$.

\para{\QTT\ decomposition as a hierarchical adaptive filter} To gain
further intuition about the compression of function samples using the
\QTT\ decomposition, and the interpretation of the cores
$\tensor{G}_i$, it is instructive to consider how the decomposition
operates step by step in \pr{alg:TT-dec}.
The \twod\ version of the compression algorithm is illustrated in
\pr{fig:TT_comp}; the structure of decomposition shown in the lower
part of the figure is identical for all dimensions.

\begin{figure}[!tb]
  \centering
  \hspace*{-40pt}
  \beginpgfgraphicnamed{TTM_compression}%
  \tikzsetnextfilename{tikz-external-TTM_compression}%
  \scalebox{1}{\subimport{figs/}{TTM_compression.pgf}}%
  \endpgfgraphicnamed%
  \mcaption{fig:TT_comp}{Steps for the \QTT\ decomposition given in
    \pr{alg:TT-dec} for a matrix with d=3}{ The source and target
    trees for this matrix have $n=m=2$ and $n_1=m_1=4$ points in each
    leaf box.
    For clarity, the colors of the blocks in the unfolding matrices
    match that of the blocks in the tree.
    The main steps of the algorithm are
    \begin{inparaenum}[(i)]
    \item computing a low rank decomposition $\laT{U}_k \laT{V}_k$
      for the corresponding unfolding matrix $\laT{M}_k$;
    \item taking $\laT{U}_k$ (in orange) as the \ordinal{k} \QTT\ core;
      and
    \item using the right factor $\la{V}_k$ as a matrix in
      level $k+1$ to form $\laT{M}_{k+1}$.
    \end{inparaenum}
    If the low rank decomposition used is interpolatory (as implied in
    the figure by the uniform subsampling of the tree nodes), this
    process corresponds to finding a hierarchical basis of matrix
    block entries.
  }
\end{figure}
At the first step, the \ordinal{j} column of the matrix $\laT{M}_1$,
$\la{c}_j$, consists of $n_1$ samples from the finest-level domains
indexed by $i_1$.
The low-rank factorization of this matrix with rank $r_1$ can be
viewed as finding a basis of $r_1 \leq n_1$ vectors forming matrix
$\laT{V}_1$, such that all vectors $\la{c}_j$ can be approximated by
linear combinations of this set of row vectors with a given accuracy.
If this decomposition is interpolatory, this corresponds to picking a
set of rows in $\laT{M}_1$ (that is, subsampling each leaf domain in
the same way), such that remaining samples can be interpolated from
these using \emph{the same} $(n_1-r) \times r$ interpolation operator.

At the next step, the subvectors for each tree node at the coarser
level 2, are arranged into vectors, which form columns of the new
matrix $\laT{M}_2$, and compressed in the same manner.
Thus, if interpolatory decomposition is used, each step of the process
can be viewed as finding the optimal subsampling of the previous level
and an interpolation matrix. The cores $\tensor{G}_i$ correspond to the
interpolation operators.
We note that in this context, matrix compression is treated as
compression of a two-dimensional sampled function.

This view of the algorithm provides intuition for the key difference
between a \QTT-based approach versus other types of matrix compression
algorithms, such as wavelet-based, \abbrev{HSS}, or $\Hmat$-matrix
algorithms.
Wavelet methods do not use adaptive filtering at all; they
perform best on sampled functions $f$ which are in the span of the
basis (e.g., linear).
In these cases, we observe that the size of the compressed
representation \emph{does not depend on the sampling resolution $N$,}
as all samples can be generated by standard wavelet refinement
operators from the fixed number of basis function coefficients needed
to represent $f$ precisely.
Because \QTT\ filters are computed adaptively, it can achieve extreme
compression ratios for various classes of functions without building
suitable filters analytically.

In the case of \abbrev{HSS} and $\Hmat$-matrix methods, the compressed
form is computed adaptively; however, this is done in a
divide-and-conquer manner: at every refinement level, a set of blocks
representing interactions at a sufficiently far distance is
compressed, but each block is compressed independently; in contrast,
\QTT\ compresses all interaction blocks at a level at once, achieving
additional gains due to block similarity.

\begin{remark}\label{rem:interlace}
  The choice of order used to quantize a vector of samples, is very important.
  Formally we can start with a suitably-sized unorganized sequence of
  samples of a function and tensorize it. Implicitly, this defines a
  hierarchical partition of this set of samples; reordering the input
  vector changes the partition, and gives a different tensor, with a
  one-to-one correspondence between different partitions and
  permutations of the elements of the input vector.

  The ranks of \QTT\ factorizations of the resulting tensors strongly
  depend on the choice of permutation.
  A bad choice (e.g. with distant points grouped in leaf nodes) may
  yield large \QTT\ ranks.
  To obtain good compression, the ordering needs to represent a
  geometrically meaningful partition, as outlined above.
\end{remark}

\section{\QTT\ ranks of integral equation operators\label{sec:integraleq}}
In this section, we present an analysis of ranks of the \QTT\
representation of matrices obtained from integral kernels.
As indicated in \pr{sec:intro}, the integral equation formulation of
{\PDE}s often involve a kernel $K(r)$ with a singularity at $r=0$, and
are usually of the form
\begin{equation}
  a\sigma(x) + \int_{\Omega} {b(x)K(||x-y||)c(y)\sigma(y) \d \Omega_y}
  = f(x), \tag{\ref{eq:basic-integral}}
\end{equation}
where $\Omega$ is a domain in $\mathbb{R}^D$ for $D=1,2,3$ (either a
boundary or a volume).
After discretization of the integral equation, e.g., using the
\nystrom\ method with an appropriate quadrature, one obtains a linear
system of the form
\begin{equation}
  \sum_{j=1}^{N} \Big[a \delta(x_i - y_j) + b(x_i)K(||x_i -
    y_j||)c(y_j) w_j \Big] \sigma(y_j) = f(x_i),
  \label{eq:nystrom}
\end{equation}
where $w_j$ denotes the quadrature weight and $x_i, y_j$ are
collocation points. We can write this in matrix form as
\begin{align}
  \laT{A\sigma} &= \la{f},
  \tag{\ref{eq:Asigma}}
\end{align}
where $\laT{A} \defeq a \laT{I} + \laT{B} \laT{K} \laT{W} \laT{C}$, in
which $\laT{B}$, $\laT{C}$ and $\laT{W}$ are diagonal matrices with
entries $b(x_i)$, $c(y_j)$ and $w_j$, respectively.

Note that the rank behavior discussed in this section is independent
from the choice of the quadrature.
In the examples of \pr{sec:results}, we use a first-order punctured
trapezoidal rule \cite{marin2014} for the volume integral and a
high-order singular quadrature using spherical harmonics for surfaces
\cite{graham2002fully}.
The surface quadrature uses trapezoidal points and weights in the
longitude direction and Gauss--Legendre points and weights in the
latitude direction.
More details can be found in \cite{rahimian2015} and \cite[Chapters
  4.3 and 18.11]{boyd1999}.

To further understand the rank behavior of the \QTT\ decomposition, we
explore the relationship between the hierarchical low-rank structure
exploited by \FMM, $\Hmat$, or \abbrev{HSS} matrices and the
matrix-block low rank structure exploited by the \QTT\ decomposition.

\subsection{\QTT\ for samples of an integral kernel \label{ssc:TTkernel}}
Consider the matrix $\laT{A}$ in \pr{eq:Asigma}, for a set of $M$
target $\{ x_i \}$ and $N$ source $\{ y_j \}$ points. As outlined
above, the entries of $\laT{A}$ are samples of a function given by the
integral kernel, with domain $D = \Omega \times \Omega$.
Thus, we may effectively apply the tensorization and
\QTT\ approximation procedures outlined in \pr{ssc:TTfunction} given a
hierarchical partition of $D$.

Let target and source trees on $\Omega$ be denoted by
$\tree{T}_\lbl{trg}$ and $\tree{T}_\lbl{src}$.
Further, assume both trees are of depth $d$ and that the number of
children at all non-leaf levels of the trees is $m$ and $n$
respectively.
There are respectively $m_1$ and $n_1$ points in each target and
source tree leaf node, making the total numbers of points are $M = m_1
m^{d-1}$ and $N = n_1 n^{d-1}$.

A partition of $D$ may be thus obtained by considering $\tree{T}$ to
be the product tree $\tree{T}_\lbl{trg} \times \tree{T}_\lbl{src}$,
whose nodes at each level consist of pairs of source and target
nodes.
At a given level $\ell$ (recall that levels are numbered starting at
the finest) each node corresponds to a matrix block with row indices
corresponding to a node of $\tree{T}_\lbl{trg}$ and column indices of
a node in $\tree{T}_\lbl{src}$, indexed by integer coordinate pairs
$(i_{k},j_{k})$ with $k\le \ell, i_k\le m_k, j_k\le n_k$.
Equivalently, we can consider block integer coordinates $b_{k} \in
\{1,\cdots,m_{k}n_{k}\}$ for $\tree{T}$ such that $b_{k} =
\overline{i_{k}j_{k}}$.

We can then apply the \QTT\ decomposition to the corresponding
tensorized form of $\laT{A}$, $\tensor{A}_{\tree{T}}$, a
\mbox{$d$-dimensional} tensor with entries defined as
\begin{equation}
  \tensor{A}_{\tree{T}}(b_1,b_2,\ldots,b_d) =
  \tensor{A}_{\tree{T}}(\overline{i_1j_1},
  \overline{i_2j_2},\ldots,\overline{i_dj_d}) =
  \laT{A}(\overline{i_1i_2 \cdots i_d}, \overline{j_1j_2 \cdots j_d})
\end{equation}
and obtain a \QTT\ factorization $\ttA[A]$.
Each core of $\ttA[A]$, $\tensor{G}_{k}(\alpha_{k-1}, \overline{i_kj_k},
\alpha_{k})$, depends only on the pair of source and target tree
indices at the corresponding level of the hierarchy.
For matrix arithmetic purposes, such as the matrix-vector product
algorithms in \pr{ssc:TTmatvec}, the cores are sometimes reshaped as
$m_k \times n_k$ matrices parametrized by $\alpha_{k-1}$ and
$\alpha_k$.

Current fast solvers exploit the fact that matrix blocks representing
interactions between well-separated target and source nodes at a given
level are of low numerical rank.
One can interpret $\tree{T}$ as a hierarchy of matrix blocks, and in
\pr{ssc:TIkernels} and \ref{ssc:NTIkernels} we show that the
\QTT\ structure for this hierarchy can be inferred from the standard
hierarchical low rank structure.

In \pr{fig:TT_comp}, we illustrated the \QTT\ decomposition algorithm
applied to a matrix $\laT{A}$ for binary source and target trees
($n=m=2$) with depth $d=3$ and four points in the leaf nodes $n_1 = m_1 =
4$, implying $N=M=16$.
In this example, the tree $\tree{T}$ is equivalent to a matrix-block
quadtree.

\subsection{Translation invariant kernels \label{ssc:TIkernels}}
If $b \equiv c \equiv 1$ in \pr{eq:basic-integral}, then the integral
kernel becomes translation-invariant in $\mathbb{R}^D$.
We assume that the domain $\Omega$ is a box in $\mathbb{R}^D$ sampled
on a uniform grid; for matrix $\laT{A}$ in \pr{eq:Asigma}, this
implies that a matrix sub-block will be invariant under translation of
both source and target points.

We begin by recalling a standard classification for pairs of boxes on
$\tree{T}_\lbl{trg} \times \tree{T}_\lbl{src}$.
\begin{definition}
  A pair of boxes $(B_i,B_j) \in \tree{T}_\lbl{trg} \times
  \tree{T}_\lbl{src}$ is said to be well-separated if
  \begin{equation}
    \dist(B_i,B_j) \geq \max\,(\diam(B_i),\diam(B_j)).
  \end{equation}
\end{definition}

\begin{definition}
  \label{def:near-far}
  We define the far field set $\mathscr{F}_{\ell}(B_i)$ of box $B_i\in
  \tree{T}_\lbl{trg}$ as the subset of boxes in $\tree{T}_\lbl{src}$
  at level $\ell$ such that $(B_i,B_j)$ is well-separated.
  Similarly, we define its near field set $\mathscr{N}_{\ell}(B_i)$ as
  the subset which is not well-separated.
\end{definition}

In a uniformly refined tree, $\card{\mathscr{N}_{\ell}(B_i)} \le 3^D$
for all $B_i \in \tree{T}_\lbl{trg}$.
For the case of adaptive trees, it is a common practice to impose a
level-restricted refinement, bounding the number of target boxes in
the near field even when neighbors at multiple levels are considered.

For all $B_j \in \mathscr{F}_\ell(B_i)$, given a desired precision
$\acc$, standard multipole estimates \cite{greengard1987fast} show
that for a broad class of integral kernels $K$ the matrix block
$\laT{A}_{i,j}$ corresponding to the evaluation of \pr{eq:nystrom}
with $(x_i,y_j) \in B_i \times B_j$ has low $\acc$-rank
$k_{i,j}$.

In fact, for kernels that arise in integral formulations of
elliptic {\PDE}s, multipole expansions or Green's type identities may
be used to prove a stronger result: the $\acc$-rank of a matrix block
with entries evaluated at $(x,y) \in S \times T$ is bounded by
$k_{\acc}$ for any well-separated sets $S$ and $T$.
This implies that interactions between a box $B$ and any subset of its
far field have bounded $\acc$-rank.

\begin{definition}
  For a matrix $\laT{A}$ given in \pr{eq:Asigma} generated by the
  partitions of $\Omega$ corresponding to $\tree{T}_\lbl{src}$ and
  $\tree{T}_\lbl{trg}$, we say that $\laT{A}$ is \FMM-compressible if for a
  given accuracy $\acc$, any matrix sub-block $A_{S,T}$ corresponding
  to evaluation at $(x,y) \in S \times T$ for well-separated sets $S$
  and $T$ is such that $\rank_\acc (\laT{A}_{S,T}) \leq k_{\acc}$.
\end{definition}

\begin{theorem}
  \label{thm:TI_TT_theorem}
  Let $K(r)$ be a translation-invariant kernel in
  \pr{eq:basic-integral}, for a box $\Omega \subset \mathbb{R}^D$. Let
  $\laT{A}$ be the corresponding matrix in \pr{eq:Asigma}, sampled on
  a regular grid.
  If $\laT{A}$ is \FMM-compressible, then for the product tree
  $\tree{T} = \tree{T}_\lbl{src} \times \tree{T}_\lbl{trg}$, the
  tensorized $\tensor{A}_{\tree{T}}$ has bounded \QTT\ ranks
  \begin{equation}
    r = \max(r_k) \leq k_{\acc}^2 + 2D - 1.
  \end{equation}
  As a consequence, the total amount of storage required for the
  \QTT\ compressed form of $\laT{A}$ is $\O{k_{\acc}^4 \log N}$.
\end{theorem}

\begin{proof}
  As indicated in \pr{sec:background} and \pr{alg:TT-dec}, the optimal
  \QTT\ ranks are the $\acc$-ranks of the unfolding matrices.
  Consider the \ordinal{\ell} unfolding matrix
  $\laT{A}^{\ell}_{\tree{T}}$ corresponding to interactions of boxes
  at level $\ell$ of the tree $\tree{T}$.
  As mentioned in \pr{ssc:TTkernel} and \pr{rem:interlace},
  columns of $\laT{A}^{\ell}_{\tree{T}}$ are vectorized matrix blocks
  $\{\vec(\laT{A}_{i,j})\}$ comprising all boxes on level $\ell$.
  We can permute the columns to place those corresponding to
  the near-field interactions first
  \begin{equation}
    \laT{A}^{\ell}_{\tree{T}} = \left[ \laT{A}_{\tree{T}}^\lbl{near}
      \; \laT{A}_{\tree{T}}^\lbl{far} \right],
    \label{eq:lowrank_far}
  \end{equation}
  The key observation is that when sampling is the same across boxes,
  at each level, boxes are translations of a reference box and we only
  need to consider inward and outward interactions for one box per
  level.
  This is also the reason why fast solvers based on hierarchical
  structures such as \abbrev{HSS} only need to compute one set of
  matrices per level for translation-invariant operators.

  For near interactions, this means only the interactions between
  the reference box and its neighbors (including itself) are needed.
  This implies $\rank_\acc( \laT{A}_{\tree{T}}^\lbl{near}) \le
  2\card{\mathscr{N}_{\ell}(B)}-1$, since the rest of the columns in this
  subset are identical to those corresponding to the reference
  box.
  For a uniformly refined tree $\card{\mathscr{N}_{\ell}(B)} \leq
  3^D$. Considering the symmetries in the interactions between these
  $3^D$ boxes gives us
  \begin{equation}
    \rank_\acc( \laT{A}_{\tree{T}}^\lbl{near}) \le 2D-1. \label{eq:rankest_near}
  \end{equation}

  For columns in $\laT{A}_{\tree{T}}^\lbl{far}$, we use an
  interpolative decomposition (\abbrev{ID}) to compute a low rank
  approximation for the matrix blocks
  \begin{equation}
    \laT{A}_{i,j} \approx \laT{L}_i \laT{M}_{i,j} \laT{R}_j,
    \label{eq:far_lowrank}
  \end{equation}
  where $\laT{L}_i$ and $\laT{R}_j$ are interpolation matrices and
  $\laT{M}_{i,j}$ is a sub-block of $\laT{A}_{i,j}$ corresponding to
  skeleton rows and columns \cite{ID2005compression,
    martinsson2007interpolation, tyrtyshnikovcross}.
  The matrix $\laT{L}_i$ is of size $|B_i| \times k_{i,j}, \laT{M}_{i,j}$ of
  size $k_{i,j} \times k_{i,j}$, and $\laT{R}_k$ is of size $k_{i,j}
  \times |B_j|$, where $|B_i| = \prod_{i\le \ell} n_i, |B_j| =
  \prod_{j\le \ell} m_j$.
  Substituting \pr{eq:far_lowrank} for each vectorized column
  $A_{i,j}$, we have
  \begin{equation}
    \vec(\laT{A}_{i,j}) = \vec(\laT{L}_i \laT{M}_{i,j} \laT{R}_j) =
    (\laT{R}^{T}_j \otimes \laT{L}_i) \vec(\laT{M}_{i,j}).
    \label{eq:vecinterp}
  \end{equation}

  Due to translation invariance, we may construct a matrix of all
  far-field interactions with a model target box $B$, and apply an
  \abbrev{ID} to obtain an interpolation matrix $\laT{L}$ and corresponding
  row skeleton set which are valid for all boxes at level
  $\ell$.\footnote{Equivalent densities may be used to accelerate this
    computation, as it is done in \cite{CMZ14} for \abbrev{HSS}
    matrices.}
  An analogous computation for a model source box may be used to obtain
  $\laT{R}$ and the corresponding column skeleton set.
  The ranks of $\laT{L}$ and $\laT{R}$ are bounded by $k_{\acc}$, by
  assumption.
  Consequently, the pre-factor on the vectorized interpolation formula
  in \pr{eq:vecinterp} is the same for all far-interaction blocks:
  \begin{equation}
    \laT{A}_{\tree{T}}^\lbl{far} = (\laT{R}^{T} \otimes
    \laT{L})\laT{M}_{\tree{T}}^\lbl{far},
  \end{equation}
  defining $\laT{M}_{\tree{T}}^\lbl{far}$ as the matrix with columns
  $\vec(\laT{M}_{i,j})$.
  This gives us a low rank decomposition of
  $\laT{A}_{\tree{T}}^\lbl{far}$ with bounded rank
  \begin{equation}
    \rank( \laT{A}_{\tree{T}}^\lbl{far}) \leq k_{\acc}^2. \label{eq:rankest_far}
  \end{equation}

  Combining the bounds in \pr{eq:rankest_near} and
  \pr{eq:rankest_far}, the rank of the unfolding matrix is bounded by:
  \begin{equation}
    \rank_\acc( \la{A}_{\tree{T}}^\ell ) \leq \rank_\acc( \la{A}^\lbl{near}_{\tree{T}} ) +
    \rank_\acc( \la{A}^\lbl{far}_{\tree{T}} ) \leq k_{\acc}^2 + 2D -
    1. \label{eq:rankest_tot}
  \end{equation}
  Since the number of cores is proportional to $\log N$, the storage
  for all the \QTT\ cores is $\O{k_{\acc}^4 \log N}$.
\end{proof}

\begin{remark}
  For an interval ($D=1$), due to translation invariance,
  $M_{i,j}=M_{i-j}$ and the matrix encoding far-field interactions
  between skeleton sets is block-Toeplitz (as the original matrix
  $\laT{A}$ is).
  This implies $\laT{M}_{\tree{T}}^\lbl{far}$ only has $2 k_{\acc}-1$
  unique columns, reducing the total \QTT\ rank bound to $2 k_{\acc} +
  2D - 2$, bringing the storage requirements for \QTT\ cores to $\O{k_{\acc}^2 \log N}$.
  Experiments with unfolding matrices for $D=2,3$ and observation of
  the resulting \QTT\ ranks and column basis using an \abbrev{ID} (as
  schematically depicted in \pr{fig:TT_comp}) suggest a similar rank bound
  (with linear dependence on $k_{\acc}$) is true for $D>1$.
\end{remark}

\subsection{Non-translation invariant kernels \label{ssc:NTIkernels}}
The (discrete) integral operator can be \NTI\ either due to nontrivial
$b(x)$ or $c(y)$\Mdash/e.g., in Lippmann--Schwinger,
Poisson--Boltzmann, or other variable coefficient elliptic
equations\Mdash/or due to the geometry of the discretization (in the
sense that $\tree{T}_\lbl{src}$ and $\tree{T}_\lbl{trg}$ correspond to
\NTI\ partitions).
For the former case, we can also conclude \QTT\ ranks are bounded as a
direct corollary of \pr{thm:TI_TT_theorem}.

\begin{corollary}
  \label{cor:NTI_TT_theorem}
  Let $K(r)$ be a translation-invariant kernel in
  \pr{eq:basic-integral}.
  Let $\laT{A} \defeq a \laT{I} + \laT{BKWC} $ be the
  corresponding discretization matrix, sampled on a translation-invariant grid.
  Let $\laT{KW}$ be \FMM-compressible, with \QTT\ ranks bounded by the
  constant $r_{\acc}^\laT{K}$.
  Further, assume $b(x)$ and $c(y)$ both admit compact
  \QTT\ representations with ranks respectively bounded by
  $r_{\acc}^\la{b}$ and $r_{\acc}^\la{c}$.
  Then the tensorized operator
  $\tensor{A}_{\tree{T}}$ on the product tree $\tree{T} =
  \tree{T}_\lbl{src} \times \tree{T}_\lbl{trg}$ has bounded
  \QTT\ ranks
  \begin{equation}
    r_{\acc}^\laT{A} = \max_{k=1,\dots,d}(r_k) \leq r_{\acc}^\la{b}
    r_{\acc}^\laT{K} r_{\acc}^\la{c} + 1.
  \end{equation}
\end{corollary}

\begin{proof}
  By \pr{thm:TI_TT_theorem}, we know that the tensorized version of
  the \FMM-compressible operator $\laT{KW}$ on $\tree{T}$ has bounded
  \QTT\ ranks.
  Regarding $b(x)$ and $c(y)$, by assumption, they are smooth,
  non-oscillatory functions.
  As indicated in \cite{khoromskij2009dlog}, the fact that
  exponential, trigonometric and polynomial functions admit exact
  low-rank \QTT\ representations implies the ranks of
  \QTT\ representations of $b$ and $c$ will be bounded by constants
  $r_{\acc}^\la{b},r_{\acc}^\la{c}$ depending on target accuracy
  $\acc$.

  Further, we can readily observe that, the diagonal matrices
  $\laT{B}$ and $\laT{C}$ have \QTT\ structure essentially the same as
  the \QTT\ structure of $\laT{b}$ and $\laT{c}$.
  Looking at the first unfolding matrix of the tensorized $\tensor{B}$,
  if we ignore columns with all zero elements,
  \begin{equation}
    \laT{B}^{1}_{\tree{T}}(\overline{i_1j_1},\overline{i_2j_2 \dots
      i_dj_d}) = \bbm 1 & 0 \\ 0 & 0 \\ 0 & 0 \\ 0 & 1 \ebm \bbm
    b(x_1) & b(x_3) & \dots & b(x_{N-1}) \\ b(x_2) & b(x_4) & \dots &
    b(x_N) \ebm,
  \end{equation}
  where the second factor on the right-hand-side is
  $\la{b}^1_\tree{T}$, the first unfolding matrix of $\la{b}$.
  Hence, the \QTT\ ranks of both decompositions are the same.

  Following the structure of the \QTT\ matrix-matrix product
  algorithm, which is analogous to the \QTT\ compressed matrix-vector
  product in \pr{ssc:TTmatvec}, the ranks of the tensorized form of
  $\laT{BKWC}$, before any rounding on the \QTT\ cores is performed, is
  equal to the product of the corresponding ranks (as the new
  auxiliary indices are a concatenation of those of each factor).

  We can thus bound the rank $r_k$ of each core of the matrix $\la{A}$ by
  the product of the corresponding ranks of $\laT{B}$, $\laT{KW}$, and
  $\laT{C}$.
  Adding an identity matrix $a\laT{I}$, which is of rank $1$ in
  \QTT\ form (a Kronecker product of identities), adds $1$ to
  this bound.
  Taking a maximum over all \QTT\ ranks, we obtain the desired bound
  \begin{equation}
    r_{\acc}^\laT{A} = \max_{k=1,\dots,d}\,(r_k) \leq r_{\acc}^\la{b}
    r_{\acc}^\laT{K} r_{\acc}^\la{c} + 1.
  \end{equation}
\end{proof}

We note that some kinds of corrected quadratures might introduce slight non-translation invariance to the operator $\laT{KW}$. However, these may generally be framed as sparse, banded perturbations of a translation invariant operator, and as such, the assumption in \pr{cor:NTI_TT_theorem} that $\laT{KW}$ have bounded \QTT\ ranks remains valid. 

\para{Non-translation invariance due to complex geometry} The general
case that is less amenable to analysis is when the operator is \NTI\ due
to $\Gamma$, often because the geometry of $\Gamma$ makes it
impossible to partition it into a spatial hierarchy of translates.
This is indeed the general case for linear systems coming from
boundary integral equations defined on curves in $\mathbb{R}^2$ or
surfaces in $\mathbb{R}^3$.

From our experiments with boundary integral operators defined on
smooth surfaces in $\mathbb{R}^3$, which we present in
\pr{ssc:bie-lattice}, we observe that \QTT\ ranks are still bounded or
slowly growing with problem size $N$, although they are generally much
larger than the ranks of the translation-invariant volumetric problem
with the same kernel in $\mathbb{R}^3$.

Although further analysis and experimentation are needed, we
conjecture that for boundary integral kernels that are
translation-invariant in the volume, the rank of far-field
interactions will remain bounded.
Near- and self-interactions are evidently dependent on surface
complexity, although we expect that if the discretization is refined
enough for a smooth surface, a relatively small basis of columns of
$\cA[A]_{\tree{T}}^\lbl{near}$ may still be found.

\section{\QTT-compressed preconditioners \label{sec:inversion}}
In this section, we describe two essential components of a \QTT-based
solvers: the computation of an approximate inverse of a matrix
(i.e. the preconditioner) in the \QTT\ format and the efficient
application of a \QTT-compressed operator to a vector.
When using \QTT\ compressed inverse as preconditioner within a Krylov
solver, we use the \FMM\ for the accurate application of the matrix
itself.

\subsection{\QTT\ matrix inversion\label{ssc:inverse}}
We use matrix inversion algorithms that are modifications of
\abbrev{DMRG} (Density Matrix Renormalization Group) and \abbrev{AMEN}
(Alternating Minimal Energy)  \cite{oseledets2012solution,
  dolgov2013amen1, dolgov2013amen2}.
These \QTT\ inversion methods provide efficient ways to directly compute the
\QTT\ decomposition of $\laT{A}^{-1}$ given the \QTT\ decomposition of
$\laT{A}$.
The approximate inversion schemes have a common starting point where
they consider the matrix equations $\laT{AX} = \laT{I}_N$ or $\laT{AX
  + XA} = 2\laT{I}_N$ and extract a \QTT\ decomposition for $\laT{X}$,
the inverse of $\laT{A}$.
If we vectorize $\laT{AX}=\laT{I}_N$ using the identity for products
of matrices, $\laT{\vec(ABC)=(C^T\otimes A)\vec(B)}$, we obtain
\begin{equation}
  (\laT{I}_N \otimes \laT{A}) \vec(\laT{X}) = \vec(\laT{I}_N).
  \label{eq:inverse}
\end{equation}
Given an initial set of cores $\{ \tensor{W}_k \}_{k=1}^{d}$ with the
corresponding \QTT\ ranks $\{ \rho_k \}_{k=1}^{d}$ for $\laT{X}$,
fixing all but $\tensor{W}_{k}$, \pr{eq:inverse} turns into a reduced
linear system, with a matrix of size $n_{k}\rho_{k-1}\rho_k \times
n_{k}\rho_{k}\rho_{k-1}$.

\abbrev{DMRG} and \abbrev{AMEN} minimization methods compute the
cores of $\laT{X}$ iteratively.
These methods start from an initial guess for the inverse in the
\QTT\ form, and proceed to solve each of the local systems in a
descent step towards an accurate \QTT\ decomposition for
$\laT{X}$.
Since the \QTT\ ranks of the inverse are not known a priori, what
distinguishes each inversion algorithm is the strategy employed to
increase the ranks of the cores as needed to accelerate convergence to
an accurate representation of the inverse.
We include further details about the \QTT\ inversion process and the
algorithms mentioned above in \pr{apx:appendix_TT-filter} and
\pr{apx:appendix_inversion}.

We note that even if \QTT\ ranks of a matrix $\laT{A}$ are small,
except for the case where the maximum rank $r_\la{A}$ is $1$, there
are no guarantees that the \QTT\ ranks of $\laT{X}$ will be small.
In \cite{tyrtyshnikov2010tensor}, for $r_\la{A}=2$, it is proven that
$r_\la{X} \leq \sqrt{N}$, and this inequality is shown to be sharp.
Nonetheless, for the integral kernels considered in this work,
extensive experimental evidence, \pr{sec:results}, shows that the
maximum ranks of forward and inverse operators are within a small
factor of each other.

\para{Computational complexity\label{ssc:inv-complexity}} Most of the
computational cost in the inversion algorithm lies in solving the
local linear systems until the desired accuracy in the
\QTT\ approximation for the inverse is achieved.
In \cite{dolgov2013amen1,dolgov2013amen2} these algorithms are shown to exhibit linear
convergence similar to that of the \abbrev{AMEN} compression
algorithm discussed in \pr{sec:background}, and so the number of cycles through the cores of $\laT{X}$
is typically controlled by its maximum \QTT\ rank for the desired
target accuracy.

Let the \QTT\ ranks of $\laT{A}$ and $\laT{X}$ be bounded by
$r_\la{A}$ and $r_\la{X}$, respectively, and $n$ denote the tensor
mode size.
The size of local systems is then bounded by $n r_\la{X}^2 \times n
r_\la{X}^2$, implying \O{r_\la{X}^6} cost of direct inversion for
each local system.
Using an iterative method to solve local systems, the complexity for
well-conditioned matrices goes down to $\O{r_\la{X}^3r_\la{A} +
  r_\la{X}^2r_\la{A}^2}$, i.e., the cost of applying
the associated matrix in the tensor form.
Since a system is solved for each of the $d$ cores of $\laT{A}$ and
$\laT{X}$, an estimate for the complexity of the whole algorithm is
$\O{(r_\la{X}^3r_\la{A} + r_\la{X}^2r_\la{A}^2) \log N}$.

\para{Preconditioning local systems} While the \abbrev{AMEN} and
\abbrev{DMRG} solvers work well for a range of examples, in many
integral-equation settings, a modification is required to ensure fast
convergence.

The condition number of the
original linear system directly affects the performance of iterative
solvers used to solve the local systems outlined above.
When the original linear system is not well-conditioned, e.g., due to
complex geometry \cite{Quaife2013a}, it is necessary to precondition
the local solves so that the performance of the inversion algorithm
does not degrade: we need to precondition the computation of the
preconditioner!

The matrices in each local system have tensor structure\Mdash/as
described in \cite{dolgov2013amen1,dolgov2013amen2}, and
\pr{apx:appendix_inversion}\Mdash/that can be exploited to construct
preconditioners for each of these local systems.
Block-Jacobi preconditioners are the most obvious solution, and are
available in the TT-Toolbox \cite{toolboxtt}.
However, we found them to be ineffective for matrices obtained from
integral equation formulations.

We propose a strategy based on a global preconditioner for the system
in \pr{eq:inverse}.
Letting $\la{M}$ denote a right preconditioner (with easy to compute
and low-rank \QTT\ representation), one can rewrite \pr{eq:inverse} in
preconditioned form as
\begin{equation}
  \vec(\laT{AMY}) = (\laT{I}_N \otimes \laT{AM}) \vec(\laT{Y}) =
  \vec(\laT{I}_N),
  \label{eq:precond_mateq}
\end{equation}
with $\laT{A}^{-1} = \laT{X} = \laT{MY}$. The conditioning of the
local systems to determine each core of $\laT{Y}$ thus depends on the
conditioning and \QTT\ representation of $\laT{AM}$. There are
multiple choices available for constructing $\laT{M}$.

In our experiments with boundary integral equations in
\pr{ssc:bie-lattice}, our integration domain consists of a collection
of disjoint surfaces with spherical topology. In this case, we opted
for a block-diagonal ``approximate'' preconditioner constructed by
replacing each surface with a sphere and analytically inverting the
self-interactions blocks (diagonal operator in spherical harmonics
basis).  We then use the \QTT\ compression algorithm to approximate
this block-diagonal system with $\ttA[M]$, compute the fast
\QTT\ product $\ttA[A]\ttA[M]$ and solve \pr{eq:precond_mateq} using
the \abbrev{AMEN} or \abbrev{DMRG} algorithms.

More generally, the preconditioner $\laT{M}$ may be the inverse of a
block-diagonal or block-sparse version of $\laT{A}$ (such as the
sparsifying preconditioners in \cite{Quaife2013a, Ying2014}).


\subsection{\QTT\ matrix-vector products \label{ssc:TTmatvec}}
The second component needed by a solver or a preconditioner is a
matrix-vector product for a matrix represented in the \QTT\ format.
When the vector is compressible in the \QTT\ form (e.g., for smooth
data), it is beneficial to compress the vector in \QTT\ form and then
apply the matrix.
We outline the matrix-vector product steps for \QTT-compressed and
uncompressed vectors as discussed in \cite{oseledets2010approximation}.

\para{\QTT\ Compressed matrix-vector product}
Let $\laT{A}$ be a
matrix and $\laT{b}$ a vector with \QTT\ decompositions consisting of
cores $\tensor{G}_{k}^{\la{A}}(\alpha_{k-1}, \overline{i_{k} j_{k}},
\alpha_{k})$ and $\tensor{G}_{k}^{\la{b}} (\beta_{k-1}, j_{k},
\beta_{k})$, respectively.
Cores for a \QTT\ decomposition of the product $\laT{c=Ab}$ is
computed as
\begin{equation}
  \tensor{G}^{\la{c}}_{k}(\overline{\alpha_{k-1}
    \beta_{k-1}}, i_{k} ,\overline{\alpha_{k} \beta_{k}})
  = \sum_{j_{k}} {\tensor{G}_{k}^{\la{A}}(\alpha_{k-1},
    \overline{i_{k} j_{k}} ,\alpha_{k})
    \tensor{G}_{k}^{\la{b}}(\beta_{k-1},j_{k},\beta_{k})}.
\end{equation}
That is, each core of $\la{c}$ is computed by the contraction over the
auxiliary index $j_k$ and merging the two pairs of auxiliary indices
$(\alpha_{k-1},\beta_{k-1})$ and $(\alpha_k,\beta_k)$.
If the \QTT\ ranks for the matrix and the vector are bounded by
$r_\la{A}$ and $r_\la{b}$, respectively, the overall computational
complexity of this structured product is $\O{r_\la{A}^2 r_\la{b}^2
  \log N}$.

\para{Matrix-vector product for an uncompressed vector} Given a
\QTT\ decomposition for a matrix $\laT{A}$, the algorithm proceeds by
contracting one index at a time, applying the corresponding
\QTT\ core.
For efficiency, it is much faster to do this contraction as a
matrix-vector operation, requiring permuting the vector elements.
This product algorithm is given in \pr{alg:TT-matvec} and has the
complexity of $\O{r_\la{A}^2 N \log N}$.

\begin{algorithm}[!bt]
  \begin{center}
    \begin{algorithmic}[1]
      \STATE \textbf{Inputs:} \QTT\ decomposition $\mathtt{A}$ with cores $\tensor{G}_k$, column vector $\laT{b}$
      \STATE \textbf{Output:} vector $\laT{y = Ab}$
      \STATE Initialize $\la{y}_0 = \la{b}^T$, and $\alpha_0,\alpha_d$ as size 1 trivial indices.
      \FOR{ $k = 1$ \TO $d$}
      \STATE Permute core dimensions and reshape as a matrix of size $r_{k}m_k \times r_{k-1}n_k$:
      \begin{equation*}
      \laT{M}_k\left(\overline{\alpha_{k}i_k } , \overline{\alpha_{k-1} j_k}\right) = \tensor{G}_k\left(\alpha_{k-1} , \overline{i_k j_k} , \alpha_{k}\right)
      \end{equation*}
      \STATE Reshape $\la{y}_{k-1}$ to merge columns from the $n_k$ children of each source box $B$, indexed by $j_k$:
      \begin{equation*}
      \laT{b}_k\left(\overline{\alpha_{k-1}j_k},\overline{J^{BOX}_{k} I^{LCL}_{k-1}}\right) =
      \la{y}_{k-1}\left(\alpha_{k-1},\overline{J^{BOX}_{k-1}I^{LCL}_{k-1}}\right)
      \label{lin:colmerge}
      \end{equation*}
      \STATE Obtain data for each target children, indexed by $i_k$:
      \begin{equation*}
      \laT{\phi}_{k} = \laT{M}_k\laT{b}_k
      \label{lin:TT-matvec-op}
      \end{equation*}
      \STATE Permute $\laT{\phi}_k$ (separate rows from the $m_k$ children of a target box $B$):
      \begin{equation*}
      \laT{y}_k(\alpha_{k},\overline{J^{BOX}_{k} I^{LCL}_{k}}) =
        \la{\phi}_{k}(\overline{\alpha_{k}i_k},\overline{J^{BOX}_{k} I^{LCL}_{k-1}})
      \label{lin:rowseparate}
      \end{equation*}
       \ENDFOR
      \RETURN $\laT{y} = \laT{y}_d^{T}$
    \end{algorithmic}
  \end{center}
  \mcaption{alg:TT-matvec}{\QTT\ matrix by uncompressed vector product}{}
\end{algorithm}
Most of the work in \pr{alg:TT-matvec} is to prepare the operands for
the index contraction as a matrix-vector multiply in
\pr{lin:TT-matvec-op}.
In the context of the matrix-block tree, sequentially contracting
indices implies an upward pass through the tree, in which one level of
the hierarchy is processed at a time, eliminating one source index
$j_k$ to compute the part of the matrix-vector product corresponding
to the target index $i_k$.
This upward pass produces a series of intermediate arrays indexed by
$\{i_1, \dotsc, i_k\}$ and $\{j_{k+1},\dotsc, j_d\}$.
The first index set determines local coordinates at each box of the
target tree, and the second index set corresponds to a box index at
level $k$ of the source tree.
To reflect this, we use the following notation
\begin{equation}
  I_k^\lbl{LCL} = \overline{i_1\cdots i_k}, \quad J_k^\lbl{BOX} =
  \overline{j_{k+1}\cdots j_d}.
\end{equation}
Notice that, by definition, $I_k^\lbl{LCL} =
\overline{I_{k-1}^\lbl{LCL}i_k}$ and $J_{k-1}^\lbl{BOX} =
\overline{j_k J_k^\lbl{BOX}}$.

When \pr{alg:TT-matvec} initializes, the vector that the first core
$\tensor{G}_1$ acts upon is the first unfolding matrix of $\laT{b}$
\begin{equation}
  \laT{b}_1(j_1,\overline{j_2 \dots j_d}) = \laT{b}(\overline{j_1 \dots j_d})
\end{equation}
For each $k>1$, we reshape the result $\laT{y}_{k-1}$ from the
previous step in \pr{lin:colmerge} in order to apply the core
$\tensor{G}_k$.
For each source box $B$ at level $k$, the $n_k$ columns of size $r_k$
from its children are merged.

The reshaped core $\laT{M}_k$ consists of $m_k\times n_k$ blocks each
of size $r_{k+1}\times r_k$.
By applying it to $\laT{b}_k$, we obtain $r_{k+1}$ results for each of
the $m_k$ boxes on $\tree{T}_\lbl{trg}$ at level $k$.
In \pr{lin:rowseparate}, we separate each block row, so that the last
column indices of $\laT{y}_k$ correspond to box indices in the target
tree.

The matrix vector product in \pr{lin:TT-matvec-op} is between a matrix
of size $r_{k+1} m_k \times r_k n_k$ and a vector of size $r_k n_k
\times \frac{N}{n_k}$, and so it requires $2m_k r_{k+1}r_{k}N$
operations.
If $m_k = n_k=n$ and $r_k \leq r_\la{A}$ for all $k$, this computation is
$\O{r^2_\la{A}N}$.
Since there are $d = \log_n N$ cores, the total computational cost is
$\O{r^2_\la{A}N\log N}$.

\section{Numerical experiments\label{sec:results}}
We present the results of a series of numerical
experiments quantifying the performance of the \QTT\ decomposition and
inversion algorithms discussed in \pr{sec:background,sec:inversion}.
As our model problems, we use linear systems of equations arising from
the \nystrom\ discretization of volume and boundary integral operators
in three dimensions, \pr{eq:basic-integral,eq:Asigma}.
We consider operators with the single- or double-layer Laplace
fundamental solution as their kernel.
For each kind of operator, we construct \QTT-based accelerated solvers
and compare them to some of the other existing alternatives.

Our implementation uses Matlab TT-Toolbox~\cite{toolboxtt} for all \QTT\
computations (with our modification to \abbrev{AMEN} and \abbrev{DMRG})
and of \abbrev{FMMLIB3D} \cite{toolboxfmm} for accurate
and fast matrix-vector apply.
All experiments are performed serially on
Intel~Xeon~E--$2690$~v$2$($3.0$~GHz) nodes with $64$~\abbrev{GB} of
memory.

\subsection{Key findings}
In \pr{ssc:tt-vie}, we demonstrate that for translation and
non-translation invariant volume integral equations with
non-oscillatory kernels, the \QTT\ inversion is cost-effective for
moderate to high target accuracies ($\sci{-10} \le \acc \le \sci{-6}$).
Accordingly, we propose using the \QTT\ inversion combined with the
fast matrix-vector algorithms as a fast direct solver.

The results reveal that the maximum \QTT\ ranks for the forward and
inverse volume operators are bounded for both translation and
non-translation invariant integrals and the maximum rank of the
inverse is proportional to that of the forward matrix
(\pr{tbl:TT-3d-6,tbl:TT-3d-NTI-6}).
Having bounded ranks for these operators results in logarithmic
scaling for factorization. As mentioned in \pr{sec:intro},
state-of-the-art direct solvers for \abbrev{HSS} and other
hierarchical matrices, when applied to volume integrals in \threed,
exhibit above linear scaling, as well as significantly high setup and
storage costs, limiting their practicality.
This makes the \QTT\ an extremely attractive alternative in this
setting.
For example, for $N = \num{262144}=64^3$, solving the translation
invariant problem in \pr{ssc:vie-TI-3d} using the \HIF\ solver
for a target accuracy of $\acc = \sci{-6}$ requires a setup time of
$32$~hours, as well as $40$~\abbrev{GB} of memory.
Setting up the corresponding \QTT\ inverse takes $86$~seconds, and
requires only $2$~\abbrev{MB} of memory (See \pr{tbl:TT-3d-6}).
We emphasize that the solve times for arbitrary right-hand sides
still scale as $\O{N\log{N}}$.

In \pr{ssc:bie-lattice}, we explore the application of the \QTT\ in
inversion of matrices arising from boundary integral equations in
complex, multiply-connected geometries.
For these systems, we employ a low-accuracy \QTT\ inverse as a
preconditioner for \GMRES, a Krylov subspace iterative method.
We establish comparisons with two types of preconditioners: a simple,
inexpensive multigrid V-cycle, and a low-accuracy \HIF\ approximate
inverse.

Both \QTT\ and \HIF\ approximate inverses provide considerable reduction
in the number of iterations (\pr{tbl:TT-BIE-iters-d4c0}).
However, they differ in their setup cost and memory requirements
(\pr{fig:SlvMV-Y43-15-D4C0,fig:SlvMV-Y43-15-d4c7}).
\QTT's sublinear setup cost and very modest memory requirements make
it more and more affordable for larger problems\Mdash/less than one
\matvec\ for $N\ge\sci[2]{5}$.
Nonetheless, \QTT\ setup and apply costs are proportional to the
maximum \QTT\ rank and hence increasing by increasing $\acc_p$.
Due to this, the $\acc_p\approx\sci{-3}$ for \QTT\ strikes the right
balance between setup cost, apply time, and iteration reduction.
When tested with progressively worse conditioned systems, both
preconditioning schemes show speedups almost independent of condition
number (\pr{fig:SpUp-vs-logdst,fig:StpandMem-vs-logdst}).

Between these two solvers, we observe a trade-off in terms of
performance and efficiency: while the obtained speedups are generally
higher for the \HIF\ due to a faster inverse apply (although the
difference decreases with problem size) the modest memory footprint
and sublinear scaling of the setup cost for the \QTT\ make it extremely
efficient, allowing the solution of problems with millions of
unknowns. Hence, the choice between these and similar direct-solver
preconditioner is problem and resource dependent. For problems in
fixed geometries involving a large number of right-hand-sides, the
additional speedup provided by \HIF\ might be desirable. For problems in
moving geometries or with a large number of unknowns, \QTT\ provides an
efficient, cost-effective preconditioner that can be cheaply updated.

\subsection{Volume integral equations\label{ssc:tt-vie}}
We test the performance of the \QTT\ decomposition as a volume
integral solver in the unit box $[-1,1]^3$ with Laplace single-layer
kernel $K(r)= \frac{1}{4\pi r}$.
We discretize the integral on a regular grid with total of $N=2^d$
points and spacing $h = 2/N^{1/3}$, indexing them according to
successive bisection of the domain along each coordinate direction
(i.e., Morton ordered).
This corresponds to a uniform binary tree with $d-1$ levels.
We use a \nystrom\ discretization of \pr{eq:basic-integral} with a
first-order ($\O{h}$) punctured trapezoidal quadrature \cite{marin2014}.
We note that corrected trapezoidal quadratures of arbitrary high-order
in two and three dimensions \cite{aguilar2005high,duan2009high} based
in \cite{kapur1997high} are available.
These corrections result in a sparse, banded perturbation to the
system matrix $\laT{A}$ obtained using the trapezoidal rule.
Thus, we expect relatively small changes in the \QTT\ rank behavior
for high-order discretizations.

\QTT\ decompositions obtained for the resulting matrix and its inverse
correspond to tensors of dimension $d$ and mode sizes $m_k = n_k =
2$.
For problem sizes ranging from $N=16^3\text{ to } 256^3$, we report
compression and inversion times in \QTT\ format, maximum \QTT\ ranks,
and storage requirements for the inverse.
For the application of the \QTT\ inverse, we test both of the apply
algorithms in \pr{ssc:TTmatvec}.
We report the time it takes to apply the \QTT\ inverse to a random,
dense vector of size $N$ (denoted as ``solve'') as well as the time it
takes to compress a vector of size $N$ sampled from a smooth function
and then apply the inverse to it (denoted as ``\QTT\ solve'').
In the
experiments presented in \pr{tbl:TT-3d-6,tbl:TT-3d-NTI-6}, we
employ \QTT-compressed representations for the right-hand-side
$f(x,y,z) = \phi(x) \phi(y) \phi(z)$ with $\phi(t)=\diric(2\pi t,10)$
(Dirichlet periodic $\sinc$ function), a smooth, oscillatory function
with bounded \QTT\ ranks:
\begin{equation}
  \mathrm{diric}(t,\nu) = \begin{cases}
  \frac{\sin(\nu t /2)}{\nu \sin(t/2)} & t \neq 2 \pi k, k \in \mathbb{Z}, \\
  (-1)^{k(\nu-1)} & t = 2 \pi k, k \in \mathbb{Z}.
  \end{cases}
  \label{eq:diric}
\end{equation}

\subsubsection{Translation invariant kernels in \threed\label{ssc:vie-TI-3d}}
We first consider the translation invariant system corresponding to
the Laplace single-layer kernel, results of which are reported in
\pr{tbl:TT-3d-6}.
We set the target accuracy of algorithms to $\acc=\sci{-6}$.
For each experiment, we test the accuracy of both forward and inverse
\QTT\ compression applied to a random, dense vectors, obtaining
residuals ranging from $\sci[1.5]{-7}$ to $\sci[5.5]{-7}$ and
$\sci[1.0]{-6}$ to $\sci[1.3]{-6}$, respectively.
\begin{table}[!tb]
  \newcolumntype{M}{>{\centering\let\newline\\\arraybackslash\hspace{0pt}$}c<{$}}
  \newcolumntype{N}[1]{>{\centering\let\newline\\\arraybackslash\hspace{0pt}}D{.}{.}{#1}}
  %
  %
  \centering
  \small
  \begin{tabular}{N{3.0} N{3.2} N{3.2} M M N{1.2} N{3.2} N{3.2} N{3.2}}\toprule
    \multirow{2}{*}[-5pt]{$N$}
    &\multicolumn{2}{c}{Time~(sec)}
    &\multicolumn{2}{c}{Max Rank}
    &\multicolumn{1}{c}{\multirow{2}{*}[-5pt]{\parbox[c]{.7in}{\centering Inverse Memory~(\abbrev{MB})}}}
    &\multicolumn{3}{c}{Inverse~Apply~(sec)}
    \\\cmidrule(lr){2-3}\cmidrule(lr){4-5}\cmidrule(lr){7-9}
    &\multicolumn{1}{c}{\parbox[c]{.6in}{\centering Compress}}
    &\multicolumn{1}{c}{\parbox[c]{.6in}{\centering Invert}}
    &\text{Forward}
    &\text{Inverse}
    &
    &\multicolumn{1}{c}{\parbox[c]{.6in}{\centering Solve}}
       &\multicolumn{2}{c}{\parbox[c]{1.2in}{\centering \QTT\ Solve\\Compress~~\&~~Apply}}
    \\ \midrule
    16^3  & 2.79 & 118.19 & 103  & 144  & 2.60 & 0.07      & 0.72  & 2.41 \\
    32^3  & 2.89 & 141.27 & 106  & 125  & 2.86 & 0.76      & 1.19  & 4.87 \\
    64^3  & 4.35 & 82.07  & \099 & \097 & 2.29 & 4.41       & 4.40  & 4.79 \\
    128^3 & 5.69 & 60.23  & \090 & \074 & 1.68 & 29.31    & 27.52 & 3.80 \\
    256^3 & 6.04 & 33.53  & \080 & \057 & 1.19 & 189.53  & 31.93 & 2.39 \\
     \bottomrule
  \end{tabular}
  \mcaption{tbl:TT-3d-6}{Translation invariant {\threed}
    volume Laplace kernel}{Compression and Inversion times, maximum
    {\QTT} ranks, memory requirements, and solve times for
    the {\QTT} decomposition algorithms applied to the
    {\threed} Laplace single-layer kernel.
    The problem sizes range from $N=\num{4096}\text{ to
    }\num{16777216}$, and the target accuracy is set to
    $\acc=\sci{-6}$.
    Achieved accuracies for the solve match this target accuracy
    closely, ranging from $\sci[1.0]{-6}$ to $\sci[1.3]{-6}$.
    For ``{\QTT} Solve'' we report the time required for the
    compression of the right-hand-side and the application of the
    {\QTT} inverse to the compressed vector.
    In this case, the right-hand-side is $f(x,y,z) = \phi(x) \phi(y)
    \phi(z)$ with $\phi(t)=\diric(2\pi t,10)$, defined in
    \pr{eq:diric}, which is a smooth, oscillatory functions whose
       {\QTT} ranks are observed to be bounded $(r \leq 75)$.
  }
\end{table}

\para{Rank behavior and precomputation costs} We see in
\pr{tbl:TT-3d-6} that the maximum \QTT\ rank for the system matrix
given a target accuracy is bounded, as argued in \pr{sec:integraleq}.
As noted in \pr{sec:inversion}, although we have no concrete estimate
for the rank behavior of the inverse, in all cases considered in this
work we observe that the maximum rank of the inverse is proportional
to that of the original matrix.  As was first observed in
\cite{khoromskaia2014Newton} for the Tucker decomposition, for a wide
range of volume integral kernels in $1, 2,$ and $3$ dimensions, we in
fact observe forward and inverse \QTT\ ranks tend to decrease with
$N$.

Since forward ranks are bounded, the number of kernel evaluations, and
hence the time it takes to produce the \QTT\ factorization (Compress
column in \pr{tbl:TT-3d-6}), displays logarithmic growth with $N$.
We recall from \pr{ssc:inverse} that the dominant cost in the
iterative inversion algorithms employed is the solution of local
linear systems for each \QTT\ core, whose sizes depend on the rank
distributions of $\laT{A}$ and $\laT{A}^{-1}$.
Even though additional levels (and corresponding tensor dimensions)
are added as $N$ increases, the decrease in \QTT\ ranks is substantial
enough to bring down the inversion time, as well as the storage
requirements for the inverse in \QTT\ form (Invert and Inverse Memory
columns in \pr{tbl:TT-3d-6}).
Perhaps the most outstanding consequence of this is how economical the
computation and storage of the inverse in the \QTT\ format is.
For $N=\num{16777216}=256^3$, with a target accuracy of $\acc =
\sci{-6}$, it takes only $34$~seconds to compute the inverse using
1.2~\abbrev{MB} of storage.

Performing these experiments with higher target accuracies, we observe
a proportional increase in \QTT\ ranks (e.g., $r_\la{A} \leq 250$ and
$r_{\la{A}^{-1}} \leq 600$ for $\acc = \sci{-10}$), and similar scaling of
ranks and computational costs with problem size to those reported in
\pr{tbl:TT-3d-6}.
Although higher ranks imply higher algorithmic constants for
compression, inversion, and apply steps, these costs stay reasonably
economical.

\para{Inverse apply} As noted in \pr{ssc:TTmatvec}, the application of
a matrix in the \QTT\ form has computational complexity dependent on
the structure of the operand.
If the operand is compressible in the \QTT\ format, the inverse apply
is $\O{\log N}$ and otherwise $\O{N\log N}$.
In \pr{tbl:TT-3d-6}, we report timing for both types of
right-hand-side, and confirm that in both cases experimental results
match the corresponding complexity analysis.
When the right-hand-side is compressible in \QTT\ form, we report
timings for right-hand-side compression (``Compress'') and
\QTT\ matrix-vector multiply (``Apply'') in the last two columns of
the table.

As we mentioned above, we use samples from a tensor product
of periodic sinc functions as a compressible right-hand-side, with
bounded \QTT\ ranks $(r_\la{b} \leq 75)$.
As indicated in the analysis in \pr{ssc:TTmatvec}, computation for
this inverse apply depends on both the ranks of the inverse and of the
right-hand-side.
Given rank bounds $r_{\la{A}^{-1}}$ for the matrix and $r_\la{b}$ for
the right-hand-side, the complexity is $\O{r_{\la{A}^{-1}}^2 r_\la{b}^2
  \log N}$.
Here, the decrease in the \QTT\ ranks of $\la{A}^{-1}$ brings down
the cost of the apply.

\subsubsection{Non-Translation invariant kernels in \threed\label{ssc:vie-NTI-3d}}
Here we test the ability of the \QTT\ decomposition to handle
non-translation invariant kernels by choosing $b(x)$ and $c(y)$ in
\pr{eq:basic-integral} to be Gaussians of the form
\begin{equation}
  b(x) = 1+e^{-(x-x_0)^T(x-x_0)}, \quad  c(y) = 1+e^{-(y-y_0)^T(y-y_0)},
\end{equation}
as it was done in \cite{CMZ14}. We report the results in
\pr{tbl:TT-3d-NTI-6}. We again test the accuracy of both forward and
inverse applies, obtaining residuals ranging from $\sci[1.1]{-6}$ to
$\sci[1.4]{-6}$ and $\sci[1.2]{-6}$ to $\sci[2.0]{-6}$, respectively.
\begin{table}[!bt]
  \newcolumntype{M}{>{\centering\let\newline\\\arraybackslash\hspace{0pt}$}c<{$}}
  \newcolumntype{N}[1]{>{\centering\let\newline\\\arraybackslash\hspace{0pt}\ensuremath}D{.}{.}{#1}}
  %
  \centering
  \small
  \begin{tabular}{N{3.0} N{3.2} N{3.2} M M N{1.2} N{3.2} N{3.2} N{3.2}}\toprule
    \multirow{2}{*}[-5pt]{$N$}
    &\multicolumn{2}{c}{Time~(sec)}
    &\multicolumn{2}{c}{Max Rank}
    &\multicolumn{1}{c}{\multirow{2}{*}[-5pt]{\parbox[c]{.7in}{\centering Inverse Memory~(\abbrev{MB})}}}
    &\multicolumn{3}{c}{Inverse~Apply~(sec)}
    \\\cmidrule(lr){2-3}\cmidrule(lr){4-5}\cmidrule(lr){7-9}
    &\multicolumn{1}{c}{\parbox[c]{.6in}{\centering Compress}}
    &\multicolumn{1}{c}{\parbox[c]{.6in}{\centering Inverse}}
    &\text{Forward}
    &\text{Inverse}
    &
    &\multicolumn{1}{c}{\parbox[c]{.6in}{\centering Solve}}
    &\multicolumn{2}{c}{\parbox[c]{1.2in}{\centering \QTT\ Solve\\Compress~~\&~~Apply}}
    \\ \midrule
    16^3  & 42.51 & 2510.38 & 386 & 209  & 5.33 & 0.14   & 0.73  & 5.65  \\
    32^3  & 62.67 & 1787.77 & 364 & 161  & 5.04 & 1.13   & 1.16  & 15.45 \\
    64^3  & 71.18 & 647.40  & 301 & 113  & 3.30 & 6.33   & 2.93  & 12.79 \\
    128^3 & 48.84 & 234.96  & 232 & \081 & 2.03 & 35.84  & 27.32 & 7.54  \\
    256^3 & 13.35 & 64.90   & 130 & \062 & 1.37 & 219.63 & 32.20 & 3.71  \\
    \bottomrule
  \end{tabular}
  \mcaption{tbl:TT-3d-NTI-6}{Non-translation invariant
    \threed\ volume Laplace kernel}{Compression and
    Inversion times, maximum {\QTT} ranks, memory
    requirements, and solve times for the {\QTT} decomposition
    algorithms applied to the non-translation invariant
    {\threed} Laplace single-layer kernel.
    Problem sizes range from $N=\num{4096}\text{ to } \num{16777216}$,
    and the target accuracy is $\acc=\sci{-6}$.
    Achieved accuracies for the
    solve match this target accuracy closely, ranging from
    $\sci[1.2]{-6}$ to $\sci[2.0]{-6}$.
    Timings under ``{\QTT} solve'' include compression of
    right-hand-sides obtained from the sampling of $f(x,y,z) = \phi(x)
    \phi(y) \phi(z)$ with $\phi(t)=\diric(2\pi t,10)$, (similar to
    \pr{tbl:TT-3d-6}) and the {\QTT} apply.
  }
\end{table}

For non-translation invariant kernels such as the one tested in
\pr{tbl:TT-3d-NTI-6}, the ranks of the matrix is expected to increase
as a function of the \QTT\ ranks of $b$ and $c$ ($r_\la{b},r_\la{c} \simeq
30$, in this case).
However, it's interesting to note that the ranks of the inverse do not
seem to increase much compared to the corresponding
translation-invariant case reported in \pr{tbl:TT-3d-6}.
This is reflected in the performance of both kinds of inverse applies.
Comparing the corresponding columns of
\pr{tbl:TT-3d-NTI-6,tbl:TT-3d-6}, we observe that the difference in
performance between both experiments decreases as $N$ increases.

\subsection{Boundary integral equations in complex geometries \label{ssc:bie-lattice}}
Except for simple surfaces, it is generally the case that for a given
target accuracy, applying the \QTT\ decomposition and inversion
algorithms as described in the previous sections will yield
\QTT\ ranks higher than in the volume cases described in
\pr{ssc:tt-vie}.
As indicated in \pr{sec:integraleq}, this is likely due to the loss of
translation invariance, which makes self and near interactions less
compressible.

As is the case for other fast direct solvers, the increase in
\QTT\ ranks implies higher algorithmic constants, and so it becomes
more practical to compute the \QTT\ compression and inversion at low
target accuracies and use them as robust preconditioners for an
iterative algorithm such as \GMRES.
In this section, we demonstrate the application of the
\QTT\ decomposition as a cost effective and robust preconditioner for
boundary integral equations.

\subsubsection{Experiment setup}
In order to build an example closer to boundary value problems
encountered in applications, we consider an exterior Dirichlet problem
for the Laplace equation on a multiply-connected complex domain with
boundary $\Gamma$
\begin{equation}
  \frac{1}{2}\sigma(x) + \int_{\Gamma} {D (\norm{x - y}) \sigma(y) \d
    \Gamma_y} = f(x), \tag{\ref{eq:basic-integral}}
\end{equation}
where $D(r)$ is the Laplace double-layer kernel in \threed.
We modify this kernel by adding rank one operator per surface to match
the far-field decay \cite{kress}.
For $\Gamma$, we choose a cubic lattice of closed surfaces
$\Gamma_i\,(i=1,\dotsc,q^3)$ of genus zero (spherical topology).
Examples of such shapes and their distribution are shown in
\pr{fig:model-surfaces}.
We discretize each surface using a basis set of spherical harmonics of
order $p$, and compute singular integrals using fast and spectrally
accurate singular quadratures \cite{graham2002fully}.
This setup, while being relevant to problems in electrostatics and
fluid flow (particulate Stokes flow), enables us to study the effects
of individual surface complexity as well as of interactions between
surfaces on the performance of the \QTT\ preconditioner.
\begin{figure}[!bt]
  \centering
  %
  %
  \subfigure[$(\ell,m)=(4,3)$]{\label{sfg:y4}\fbox{\includegraphics[width=.22\linewidth]{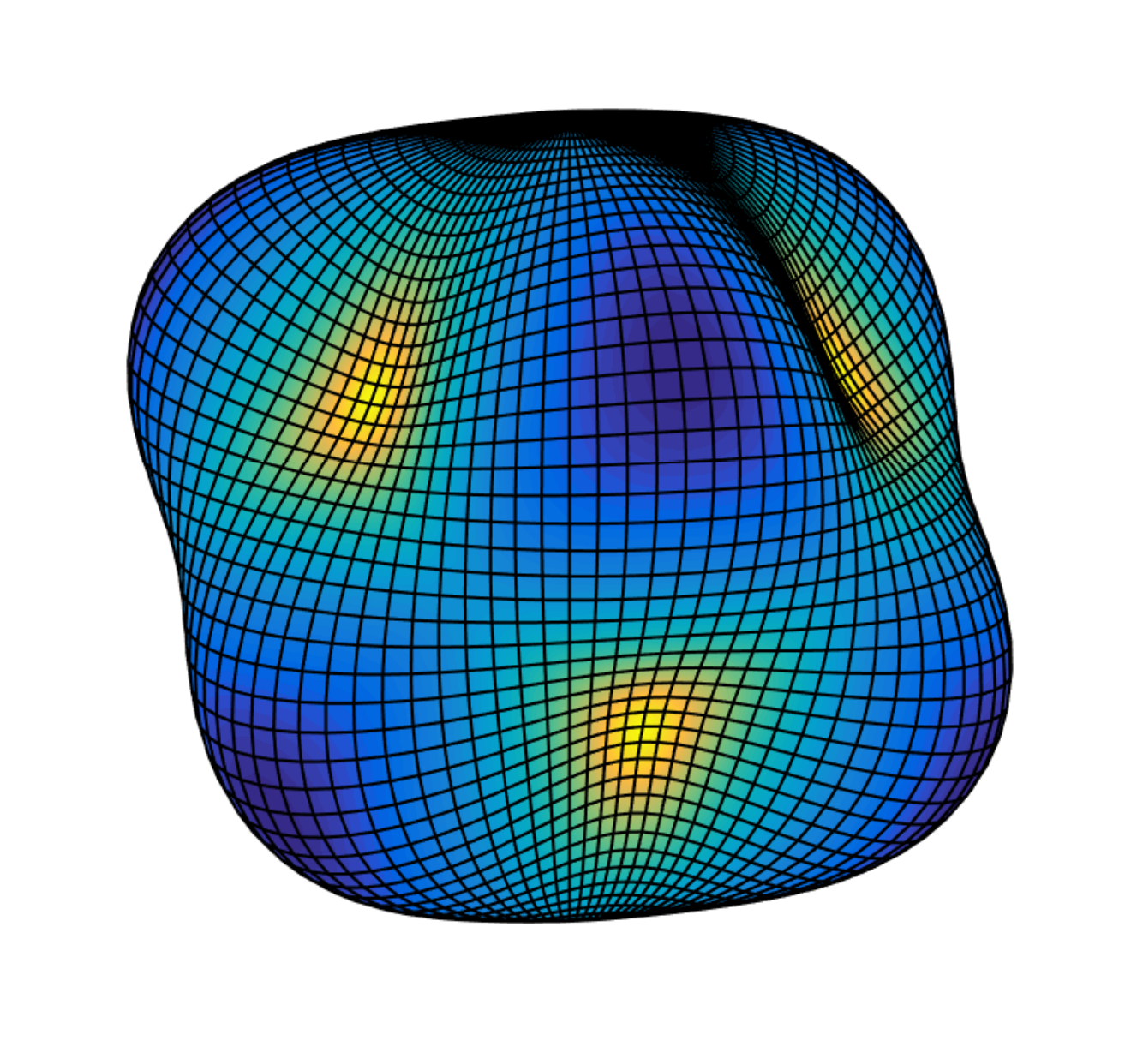}}}
  \subfigure[$(\ell,m)=(8,7)$]{\label{sfg:y8}\fbox{\includegraphics[width=.22\linewidth]{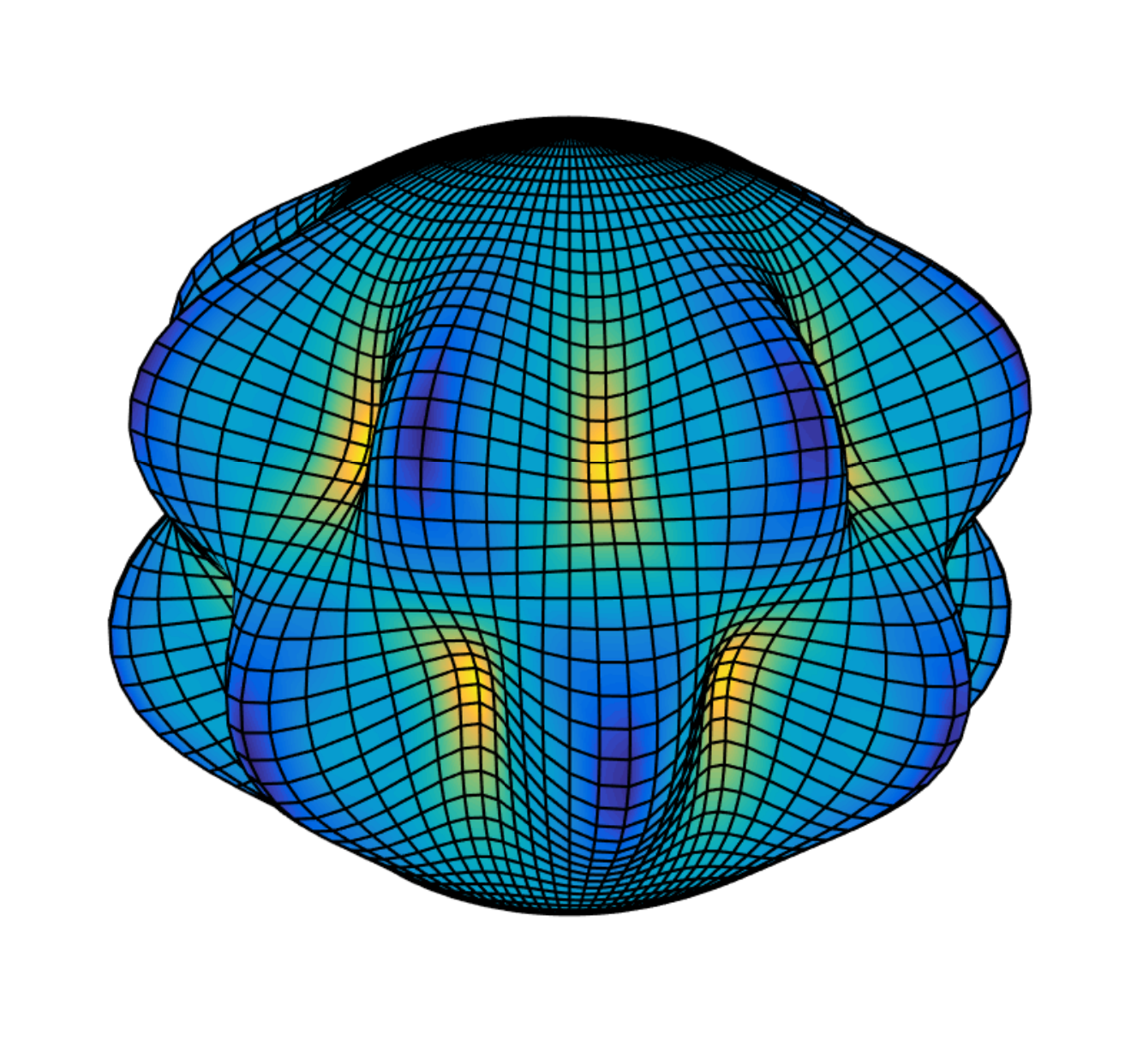}}}
  \subfigure[$(\ell,m)=(16,15)$]{\label{sfg:y16}\fbox{\includegraphics[width=.22\linewidth]{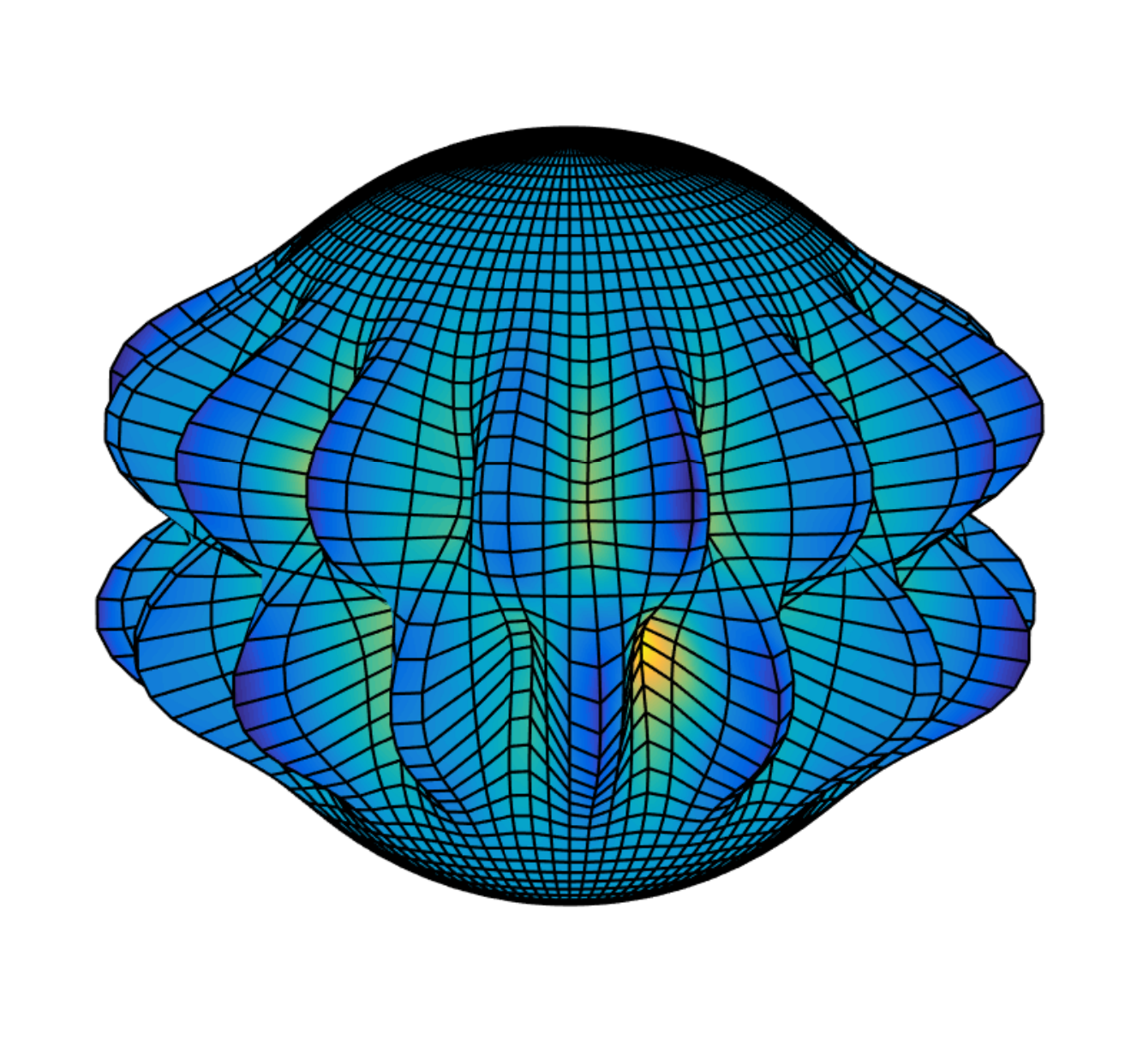}}}
  \subfigure[Irregular
    Lattice]{\label{sfg:lattice}\fbox{\includegraphics[width=.22\linewidth]{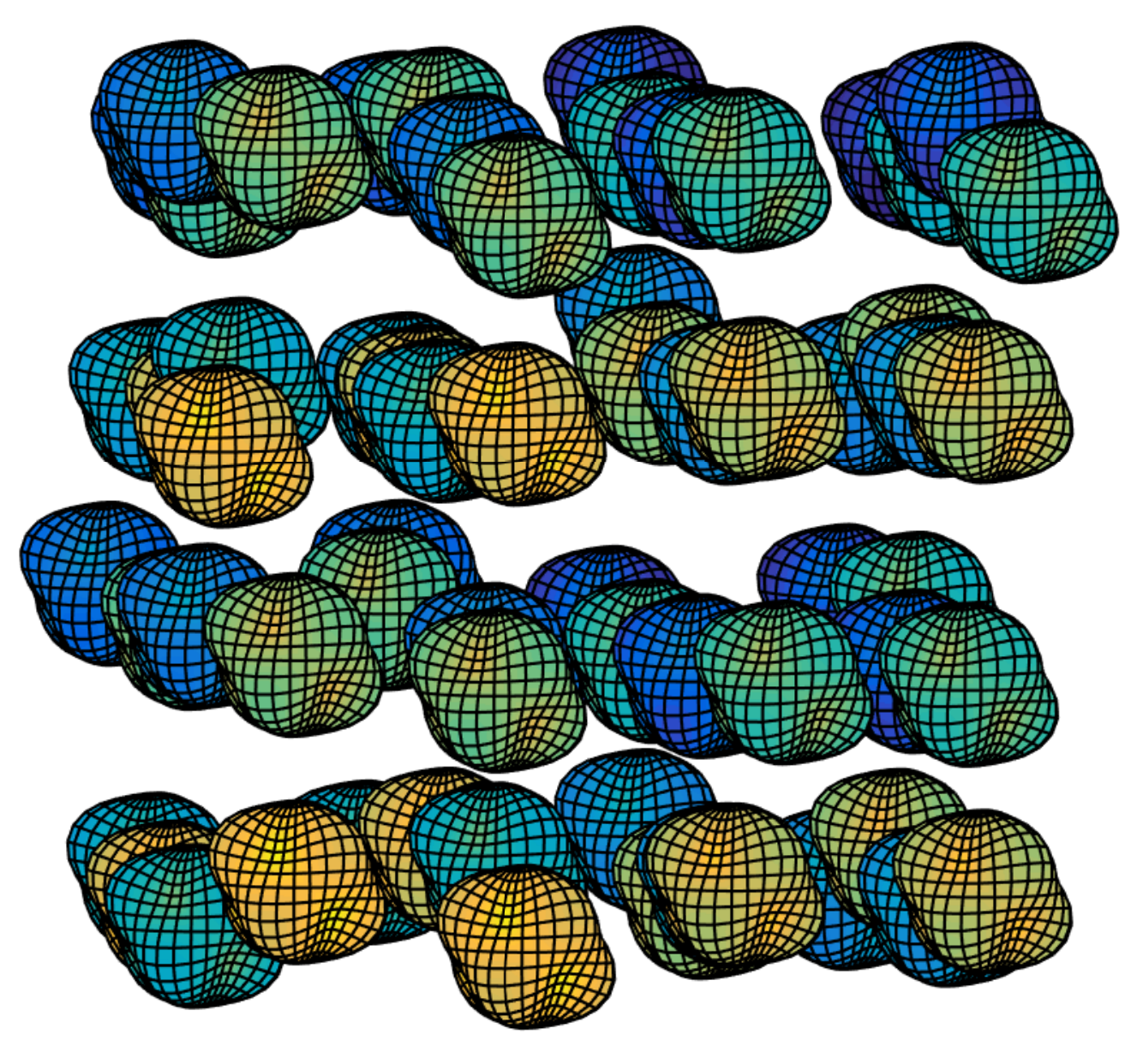}}}
  \mcaption{fig:model-surfaces}{Example of test
  geometries}{\subref{sfg:y4}--\subref{sfg:y16} Example surfaces used
      in our experiments. The chart for surfaces is given by
      $\rho(\theta,\phi) = 1 + 0.5Y_{\ell m}(\theta,\phi)$ where
      $Y_{\ell m}$ denotes the spherical harmonics.
      The color is the signed mean curvature.
      \subref{sfg:lattice} A random lattice of $64$ shapes
      (arbitrarily colored).
    }
\end{figure}

In all experiments in this section, we precondition the local systems
based on the framework discussed in \pr{ssc:inverse}.
We construct a corresponding block-diagonal system $\la{M}$ based
spheres as $\Gamma_i$ and only considering the self interaction of
each sphere.
The \QTT\ compressed form of this operator is denoted by $\ttA[M]$.
The \QTT\ preconditioner is thus the product of two \QTT\ matrices
$\ttA[M]$ and $\ttA[Y]$. This provides a significant acceleration for the inversion algorithm
by improving the conditioning of the local linear systems.
It also provides an acceleration for the resulting inverse apply, as
the resulting ranks of $\ttA[M]$ and $\ttA[Y]$ are observed to
be smaller than those of their product.

For the sake of comparison with the volume integral experiments in
\pr{ssc:tt-vie}, in \pr{tbl:TT-BIE-3} we report results for matrix
compression and preconditioner setup in \QTT\ form for $\acc_p =
\sci{-3}$, for a regular lattice of surfaces with spherical topology
and radius $\rho = 1 + 0.5Y_{4,3}$.
Unlike the cases in \pr{tbl:TT-3d-6,tbl:TT-3d-NTI-6}, maximum ranks
for both the forward matrix $\ttA[A]$ and the matrix $\ttA[Y]$ in the
preconditioner show a relatively slow but steady increase with problem
size $N$.
As a consequence, both timings and inverse memory display sublinear
growth with $N$.
We note that in terms of rank behavior and the scaling of
precomputation costs, the results presented in \pr{tbl:TT-BIE-3} are
representative of all experiments presented in this section.
\begin{table}[!tb]
  \newcolumntype{M}{>{\centering\let\newline\\\arraybackslash\hspace{0pt}$}c<{$}}
  \newcolumntype{N}[1]{>{\centering\let\newline\\\arraybackslash\hspace{0pt}}D{.}{.}{#1}}
  \newcolumntype{L}{>{$\collectcell\num}r<{\endcollectcell$}}
  %
  %
  \centering
  \small
  \begin{tabular}{L N{3.2} N{3.2} M M N{1.2}}\toprule
    \multicolumn{1}{c}{\multirow{2}{*}[-5pt]{$N$}}
    &\multicolumn{2}{c}{Time~(sec)}
    &\multicolumn{2}{c}{Max Rank}
    &\multicolumn{1}{c}{\multirow{2}{*}[-5pt]{\parbox[c]{1.0in}{\centering Inverse Memory ~(\abbrev{MB})}}}
    \\\cmidrule(lr){2-3}\cmidrule(lr){4-5}
    \multicolumn{1}{c}{}
    &\multicolumn{1}{c}{\parbox[c]{.6in}{\centering Compress}}
    &\multicolumn{1}{c}{\parbox[c]{.6in}{\centering Invert}}
    &\text{Forward}
    &\text{Inverse}
    &
    \\ \midrule
    3840    & 27.75 & 375.49 & 117  & 86~(49) & 8.56               \\
    30720   & 34.80 & 512.18 & 133  & 90~(49) & 11.73              \\
    245760  & 45.42 & 630.46 & 144  & 93~(50) & 14.25              \\
    1996080 & 49.50 & 732.21 & 150  & 95~(50) & 18.75              \\
     \bottomrule
  \end{tabular}
  \captionsetup{width=.85\linewidth}
  \mcaption{tbl:TT-BIE-3}{\QTT\ preconditioner ranks and
    precomputation costs for the benchmark case}{Compression and
    Inversion times, maximum {\QTT} ranks, and memory
    requirements for the {\QTT} decomposition algorithms for
    benchmark case on a regular cubic lattice.
    The {\QTT} inverse is a product of two {\QTT}
    matrices, a block-diagonal system for spherical surfaces $\ttA[M]$
    and the solution for the right-preconditioned system $\ttA[X]$, as
    discussed in \pr{ssc:inverse}.
    We include maximum ranks for both {\QTT} matrices,
    reporting those for block-diagonal $\mathtt{M}$ in parentheses,
    and reporting the aggregate memory requirements for both under
    ``Inverse Memory''.
    The problem sizes range from $N=\num{3840}~(n=8,p=15)\text{ to
    }\num{1996080}~(n=4096, p=15)$, and the target accuracy for all
    algorithms to construct the {\QTT} preconditioner is set
    to $\acc_p=\sci{-3}$.
    The model surface has radius $\rho = 1 + 0.5Y_{4,3}$.
  }
\end{table}

\subsubsection{Comparison with other preconditioners}
Multigrid, sparsifying preconditioners \cite{Quaife2013a, Ying2014},
and hierarchical matrix preconditioners (using \abbrev{HSS-C},
\HIF, \abbrev{IFMM} \cite{IFMM2015prec}, or $\Hmat$ inverse
compression at low target accuracies) are a few options available for
this type of problem.

We compare \QTT's setup costs (inversion time and memory
requirements), its effectiveness in terms of iteration count for the
iterative solver, and total solve time with those of multigrid and
\HIF.
We use a two-level multigrid V-cycle, considered in
\cite{Quaife2013a}, as preconditioner.
Levels in this context are defined by the spherical harmonic order
$p$.
In our experiments, we choose $p_\lbl{coarse} =\ceil{p/2}$ as the
coarse level.
We use the natural spectral truncation and padding for restriction and
prolongation operators, and for smoothing, we use Picard iteration.
The coarse-grid problems are solved iteratively using \GMRES\ with tolerance
$\acc_p$.
For \HIF\ \cite{ho2015hierarchical_ie}, we use a Matlab implementation
kindly provided to us by Kenneth~Ho and Lexing~Ying.

For elliptic {\PDE}s, multigrid provides an acceleration to the iterative
solvers with almost negligible setup costs and storage requirements;
however, its performance for integral equations is less understood.
Fast direct solvers based on hierarchical matrix compression like \HIF,
on the other hand, have significantly large setup costs.
Nevertheless, when preconditioners of the latter kind are affordable
to construct for low accuracies, they present an effective
preconditioner, reducing the number of iterations while incurring
small cost associated with the application of the preconditioner,
because of their quick apply.

Since \QTT\ algorithms also provide a hierarchical decomposition of
the inverse, we expect their performance to be similar to fast direct
solvers such as \HIF. We also anticipate that, due to its modest setup
costs, it will enable the solution of problems with a large number of
unknowns.

\subsubsection{Overview of experiments}
In order to successfully test the \QTT\ preconditioner, we first
establish a benchmark case, choosing a regular, cubic lattice of
identical translates as our geometry. We expect this to be the most
advantageous case for the \QTT, as it is able to exploit any
regularities present in the integration domain. We report our results
in terms of matrix-vector applies required for the solve
(\pr{tbl:TT-BIE-iters-d4c0}).
Since the multigrid approach proves to be limited in its
effectiveness, we concentrate on a more detailed comparison of solve
times against setup costs between \QTT\ and
\HIF\ (\pr{fig:SlvMV-Y43-15-D4C0}).

We then devise two stress tests to determine the robustness of the
\QTT\ approach. First, we eliminate the regularity in the cubic
lattice by randomly perturbing centers, radii and types of the
surfaces that constitute it. As expected, a moderate increase in
precomputation costs is observed for the
\QTT\ (\pr{fig:SlvMV-Y43-15-d4c7}).
However, except for the lowest
preconditioner accuracies, its performance in terms of iteration
counts does not degrade when compared to the benchmark.

Finally, we subject both preconditioners to a series of experiments
with progressively worse conditioning (we draw the surfaces in the
lattice close to contact).
However, for preconditioner accuracies $\acc_p \leq 10^{-2}$,
iteration counts and solve times for both preconditioners become
practically independent of distance between surfaces, providing
increasing speed-ups when compared to the unpreconditioned solve. We
display results for $\acc_p = 10^{-3}$
(\pr{fig:SpUp-vs-logdst,fig:StpandMem-vs-logdst}).

\subsubsection{Benchmark}
We test a cubic lattice of $q \times q \times q$ surfaces, for $n=q^3
\in\{8,64,512,4096\}$, discretizing each surface using spherical
harmonic basis of order $p\in\{15,24,35\}$, which requires $p+1$
collocation points in the latitude direction and $2p$ collocation
points in the longitude direction. The total number of unknowns for each problem is then $N = 2p(p+1)
n$. We make all surfaces translations of a single shape on a regular lattice with spacing of $4$, which implies
the closest distance between surfaces is slightly smaller than their diameter.
To control surface complexity, we make their radius $\rho(\theta,\phi)
= 1 + 0.5Y_{\ell m}(\theta,\phi)$, where $Y_{\ell m}$ is the spherical
harmonic function of order $(\ell,m)$.

For the matrix-vector apply, we use the \abbrev{FMMLIB3D} library
\cite{toolboxfmm} for interactions between surfaces, and spectral
quadrature for interactions within each surface \cite{graham2002fully}.
We test three surfaces of increasing complexity by setting $\ell \in
\{ 4,8,16 \}$ and $m=\ell-1$, shown in \pr{fig:model-surfaces}.

The results of the tests on this lattice are reported in
\pr{tbl:TT-BIE-iters-d4c0}.
For all problem sizes, we compare the number of matrix-vector
applications for the unpreconditioned solve with those of the
preconditioned solve with multigrid, \QTT, and \HIF.
Each approximate-inverse preconditioner is constructed for a target
accuracy of $\acc_p = \sci{-3}$ and is used within a \GMRES\ solve with a
tolerance of $\acc=\sci{-8}$ for relative residual and no restart.
\begin{table}[!tb]
  \newcolumntype{M}{>{$}c<{$}}
  \newcolumntype{N}{>{$\collectcell\num}r<{\endcollectcell$}}
  \centering
  \small
  \begin{tabular}{ MNN MMM MMM MMM MMM}\toprule
    \multirow{2}{*}[-3pt]{$p$}
    &\multicolumn{1}{c}{\multirow{2}{*}[-3pt]{$n$}}
    &\multicolumn{1}{c}{\multirow{2}{*}[-3pt]{$N$}}
    &\multicolumn{3}{c}{Unpreconditioned}
    &\multicolumn{3}{c}{Multigrid}
    &\multicolumn{3}{c}{\QTT}
    &\multicolumn{3}{c}{\HIF}
    \\\cmidrule(lr){4-6} \cmidrule(lr){7-9} \cmidrule(lr){10-12} \cmidrule(lr){13-15}
       & \multicolumn{1}{c}{} & \multicolumn{1}{c}{} & \ell=4 & 8   & 16   & 4         & 8         & 16       & 4  & 8  & 16 & 4  & 8  & 16 \\\midrule
    15 & 8                    & 3840                 & 129    & 401 & 1001 & 25~(140)  & -         & -        & 9  & 9  & 9  & 7  & 7  & 7  \\
    15 & 64                   & 30720                & 153    & 255 & 991  & 25~(140)  & -         & -        & 11 & 11 & 9  & 7  & 7  & 9  \\
    15 & 512                  & 245760               & 141    & 277 & 788  & 25~(140)  & -         & -        & 11 & 11 & 11 & 7  & 11 & 7  \\\midrule
    24 & 8                    & 9600                 & 153    & 229 & 617  & 21~(92)\0 & 37~(180)  & -        & 9  & 13 & 11 & 7  & 7  & 7  \\
    24 & 64                   & 76800                & 141    & 207 & 677  & 21~(114)  & 37~(180)  & -        & 11 & 13 & 15 & 9  & 7  & 7  \\
    24 & 512                  & 614400               & 141    & 205 & 577  & 21~(103)  & 37~(180)  & -        & 11 & 11 & 15 & 15 & 13 & 13 \\\midrule
    35 & 8                    & 20160                & 139    & 429 & 601  & 17~(16)\0 & 25~(84)\0 & 25~(140) & 9  & 11 & 11 & 7  & 7  & 9  \\
    35 & 64                   & 161280               & 143    & 187 & 651  & 17~(16)\0 & 25~(73)\0 & 25~(140) & 9  & 11 & 11 & 9  & 7  & 9  \\
    35 & 512                  & 1290240              & 137    & 193 & 595  & 17~(16)\0 & 25~(79)\0 & 25~(140) & 9  & 11 & 11 & -  & -  & -  \\\bottomrule
  \end{tabular}
  \mcaption{tbl:TT-BIE-iters-d4c0}{Matrix-vector apply count for
    the benchmark case}{Comparison of unpreconditioned \GMRES\ and
    preconditioned \GMRES\ using multigrid, \QTT, and \HIF\ preconditioners
    for the exterior \threed\ Laplace problem over a regular lattice of $n$
    model surfaces each with radius $\rho = 1 + 0.5Y_{\ell m}$ where
    $\ell= \{ 4,8,16 \}$ and $m=\ell-1$.
    Surfaces are represented in spherical harmonics basis and
    discretized with $2p(p+1)$ collocation points.
    The relative residual tolerance for \GMRES\ is set to $\acc =
    \sci{-8}$ and \GMRES\ is not restarted.
    For multigrid, the number of applies in the coarse \GMRES\ solver
    (with $\acc_p = \sci{-2}$) is reported in parentheses.
    \QTT\ and \HIF\ approximate-inverse preconditioners are constructed
    with a target accuracy of $\acc_p = \sci{-3}$.
    Empty entries correspond to experiments in which either the
    preconditioned \GMRES\ failed to converge in $100$ iterations or
    preconditioner setup costs were excessive.
    }
\end{table}

From the results in \pr{tbl:TT-BIE-iters-d4c0}, it can be readily
observed that in this configuration the condition number of the
problem, implied by number of iterations of the unpreconditioned
solve, mostly depends on the individual surface complexity and not the
number of surfaces $n$ or problem size $N$.

As we increase surface complexity (controlled by $\ell$), the average
number of matrix-vector applies goes from $140$ to $240$ to $600$,
rendering the unpreconditioned solve impractical.
The multigrid preconditioner is easy and inexpensive to setup, but, as
it is mentioned in \cite{Quaife2013a}, its performance suffers when
the geometry is not resolved in the coarse grid, i.e. $p_\lbl{coarse}$
is not large enough.
In \pr{tbl:TT-BIE-iters-d4c0}, we observe that after the individual
surface geometry is resolved in the coarse grid, multigrid leads to a
reduction in the iteration counts.
However, this requires unnecessary over-resolution in the fine grid
that is not required by the problem but by the preconditioner.

Furthermore, note that each multigrid cycle uses two matrix-vector
applies at the fine level as well as a number of applies at the coarse
level, and so the observed speedups with respect to the
unpreconditioned solve are moderate.
Additionally, for cases where both fine and coarse levels are
resolved, such as $p=35$ and $\ell=4$, the number of coarse applies is
still higher than the corresponding \QTT\ and \HIF\ direct solvers at the
level equal to the multigrid coarse level ($p=15$ and $\ell=4$).
Due to its mediocre performance, we choose not to further consider
multigrid for our comparisons, focusing instead on comparing \QTT\ and
\HIF\ direct-solver preconditioners.

As expected, both \QTT\ and \HIF\ display consistent performance across
all cases considered and they display little or no dependence on
surface complexity $\ell$, lattice size $n$, and number of
unknowns $N$.
For both approaches, the cost of an iteration is one matrix-vector
apply and one fast apply of the corresponding compressed low-accuracy
inverse.
For most examples considered, applying the preconditioner takes only a
fraction of the matrix-vector apply, allowing for considerable
speedups.

In the following section, we explore different aspects of \QTT\ and
\HIF\ preconditioners.

\subsubsection{Comparison of direct-solver preconditioners}
For a given model surface and problem size $N$, we report the setup
costs (inversion time and storage) and the solve times using both
direct solvers with inversion accuracies from $\acc_p =\sci{-1} \text{
  to } \sci{-3}$.
In order to represent wall-clock time in a more meaningful way across
experiments of different sizes, we normalize it in terms of the \O{N}
matrix-vector applies for the system matrix through \abbrev{FMM}, hereinafter
denoted by \tmatvec.
For $N>\num{500000}$, high inversion cost generally prevents us from
constructing the \HIF\ preconditioner.

In \pr{fig:SlvMV-Y43-15-D4C0}, we plot solve times against setup costs
(inversion time and storage) for both \QTT\ and \HIF\ preconditioners.
This allows us to identify distinct trade-offs in performance and
efficiency, as well as to observe their scaling with respect to $N$
and $\acc_p$.
\begin{figure}[!bt]
  \centering
  \small
  \setlength\figureheight{.35\textwidth}
  \setlength\figurewidth{.5\textwidth}
  
  \beginpgfgraphicnamed{SlvMV_Y43_D4C0}%
  \tikzsetnextfilename{tikz-external-SlvMV_Y43_D4C0}%
  \scalebox{1}{\begin{tikzpicture}
  \begin{groupplot}[%
      ,group style={rows=1,columns=2,group name=costplot}
      ,width=\figurewidth
      ,height=\figureheight
      ,ymajorgrids
      ,xmajorgrids
      ,ytick scale label code/.code={v}
      ,xmode=log
      ,xmin=0.01
      ,xmax=10000
      ,xminorticks=true
      ,ymin=0
      ,ymax=35
      ,ytick={ 5, 10, 15, 20, 25, 30, 35}
      ,legend style={legend columns=4,font=\footnotesize}
      ,legend cell align=left
    ]
    \nextgroupplot[%
      ,legend to name=costplot-legend
      ,mark=none
      ,xlabel={$\log_{10}\left({\text{Precond. Setup}}/{\tmatvec}\right)$}
      ,ylabel={Solve Time/\tmatvec}
    ]
    \addplot [color=plt-lc1,dash pattern=on 4pt off 2pt,line width=1.0pt,mark size=2.0pt,mark=o,mark options={solid}]
    table[row sep=crcr]{22.9553090669269	30.8572790185025\\
      41.8251677083952	14.000765394825\\
      108.386129356536	11.9044846331775\\
      534.125769925782	12.5951745620851\\
      4659.4793220683	12.8378771733164\\
    };\addlegendentry{TT $(N=\num{3840})$}
    \addplot [color=plt-lc2,dash pattern=on 4pt off 2pt,line width=1.0pt,mark size=2.0pt,mark=triangle,mark options={solid}]
    table[row sep=crcr]{1.98845759676409	25.4083785035274\\
      4.13068144614537	17.3843039030592\\
      12.6070284317619	16.0173322396541\\
      67.2857211549693	13.006898290961\\
      503.520953355223	12.5164168115806\\
    };\addlegendentry{TT $(N=\num{30720})$}
    \addplot [color=plt-lc3,dash pattern=on 4pt off 2pt,line width=1.0pt,mark size=2.5pt,mark=diamond,mark options={solid}]
    table[row sep=crcr]{0.244648296286241	25.2116665353281\\
      0.503623606787123	17.3192079953477\\
      1.60741202995278	15.9711797143472\\
      8.74197638981873	12.9970894736854\\
      59.9743682888103	12.6194238299524\\
    };\addlegendentry{TT $(N=\num{245760})$}
    \addplot [color=plt-lc4,dash pattern=on 4pt off 2pt,line width=1.0pt,mark size=2.5pt,mark=x,mark options={solid}]
    table[row sep=crcr]{0.0313004399535461	24.7021581443837\\
      0.0642139565931514	17.0171076537273\\
      0.204077033522939	15.7562788756599\\
      1.13338799834556	13.0087289470275\\
      10.5672977029331	13.1962476751256\\
    };\addlegendentry{TT $(N=\num{1996080})$}
    \addplot [color=plt-lc1,solid,line width=1.0pt,mark size=2.0pt,mark=o,mark options={solid}]
    table[row sep=crcr]{7.23309211477306	15.04\\
      10.9677410463647	8.55\\
      13.1422453421302	7.89\\
      17.4335784486864	7.16\\
      18.0157304511087	5.07\\
    };\addlegendentry{HIF $(N=\num{3840})$}
    \addplot [color=plt-lc2,solid,line width=1.0pt,mark size=2.0pt,mark=triangle,mark options={solid}]
    table[row sep=crcr]{13.9075094312641	15.35\\
      16.970362150202	11.54\\
      19.6934490951804	9.33\\
      21.5058979739766	7.27\\
      23.3250227979741	7.22\\
    };\addlegendentry{HIF $(N=\num{30720})$}
    \addplot [color=plt-lc3,solid,line width=1.0pt,mark size=2.4pt,mark=diamond,mark options={solid}]
    table[row sep=crcr]{84.0128011483136	17.34\\
      89.0766307464797	15.43\\
      92.3611912921938	9.83\\
      96.2215006456474	7.5\\
      100.955550030095	7.48\\
    };\addlegendentry{HIF $(N=\num{245760})$}
    \nextgroupplot[mark=none,xlabel={$\log_{10}$(Precond Memory) [MB]},yticklabels={\empty}]
    \addplot [color=plt-lc1,dash pattern=on 4pt off 2pt,line width=1.0pt,mark size=2.0pt,mark=o,mark options={solid}]
    table[row sep=crcr]{0.01019287109375	30.8572790185025\\
      0.10760498046875	14.000765394825\\
      0.6959228515625	11.9044846331775\\
      2.46687316894531	12.5951745620851\\
      8.56077575683594	12.8378771733164\\
    };
    \addplot [color=plt-lc2,dash pattern=on 4pt off 2pt,line width=1.0pt,mark size=1.3pt,mark=triangle,mark options={solid}]
    table[row sep=crcr]{0.019317626953125	25.4083785035274\\
      0.126663208007812	17.3843039030592\\
      0.93804931640625	16.0173322396541\\
      4.20570373535156	13.006898290961\\
      11.7257537841797	12.5164168115806\\
    };
    \addplot [color=plt-lc3,dash pattern=on 4pt off 2pt,line width=1.0pt,mark size=3.5pt,mark=diamond,mark options={solid}]
    table[row sep=crcr]{0.024932861328125	25.2116665353281\\
      0.14605712890625	17.3192079953477\\
      1.26962280273438	15.9711797143472\\
      5.41299438476562	12.9970894736854\\
      14.2541198730469	12.6194238299524\\
    };
    \addplot [color=plt-lc4,dash pattern=on 4pt off 2pt,line width=1.0pt,mark size=2.0pt,mark=x,mark options={solid}]
    table[row sep=crcr]{0.0297347193286361	24.7021581443837\\
      0.174186495357224	17.0171076537273\\
      1.46795654296875	15.7562788756599\\
      6.45491027832031	13.0087289470275\\
      18.7480621337891	13.1962476751256\\
    };
    \addplot [color=plt-lc1,solid,line width=1.0pt,mark size=2.0pt,mark=o,mark options={solid}]
    table[row sep=crcr]{14.6290664672852	15.04\\
      32.1344757080078	8.55\\
      46.0841903686523	7.89\\
      53.090705871582	7.16\\
      57.5647506713867	5.07\\
    };
    \addplot [color=plt-lc2,solid,line width=1.0pt,mark size=2.0pt,mark=triangle,mark options={solid}]
    table[row sep=crcr]{115.995407104492	15.35\\
      272.313339233398	11.54\\
      391.554496765137	9.33\\
      472.849166870117	7.27\\
      557.004211425781	7.22\\
    };
    \addplot [color=plt-lc3,solid,line width=1.0pt,mark size=2.5pt,mark=diamond,mark options={solid}]
    table[row sep=crcr]{943.885856628418	17.34\\
      2277.46474456787	15.43\\
      3695.83222961426	9.83\\
      4829.26739501953	7.5\\
      6236.88101959229	7.48\\
    };
  \end{groupplot}
  \node at ($(costplot c1r1.north)!.5!(costplot c2r1.north)$) [anchor=south] {\ref{costplot-legend}};
  \node at (costplot c1r1.north) [anchor=south, yshift=-.11\figureheight, xshift=.35\figurewidth]{(a)};
  \node at (costplot c2r1.north) [anchor=south, yshift=-.11\figureheight, xshift=.35\figurewidth]{(b)};
\end{tikzpicture}%
}%
  \endpgfgraphicnamed%
\\
  \mcaption{fig:SlvMV-Y43-15-D4C0}{Solve time vs. preconditioner
    setup time and memory for the benchmark case}{ Semi-log plots for
    solve times, normalized by the \matvec\ time \tmatvec, vs. the required
    inversion time, also normalized by \tmatvec, and storage
    requirements (in \abbrev{MB}) for \QTT\ and \HIF\ preconditioners for $p=15$,
    $n=\{8,64,512,4096\}$ and $\acc_p =\sci{-1} \text{ to }
    \sci{-3}$.
    Model surface has radius $\rho = 1 + 0.5Y_{4,3}$.
    Because of $\log(N)$ complexity of \QTT\ compression and inversion,
    \QTT\ scheme becomes cheaper for larger problem sizes.
    }
\end{figure}

\para{Solve time and speedup} As we increase the preconditioner
accuracy $\acc_p$, the number of iterations for the solve
decreases while the cost of applying the corresponding preconditioner
increases as \QTT\ and \HIF\ ranks increase.
Both provide considerable speedups (unpreconditioned solve takes about
$140\tmatvec$), with \HIF\ mostly yielding the highest when available.
The main reason behind this is that the \HIF\ apply has a more favorable
dependence on $\acc_p$ than the \QTT\ apply.
This is evident by the fact that the rise in preconditioner
application cost causes the \QTT\ solve time to plateau but it does not
significantly affect the speedups yielded by \HIF\ for the range of
accuracies considered. For higher accuracies ($\acc_p \lesssim \sci{-4}$), we observe a slight increase in solve times for \QTT\, as well as significant increases to setup costs for both solvers. 

\para{Inverse setup time}
\reftext[(a)]{fig:SlvMV-Y43-15-D4C0}{\figkeyword~} shows the solve
time versus the setup time as $\acc_p$ increases for different $N$.
While \HIF\ setup costs relative to \tmatvec\ increase as we increase
the problem size, logarithmic scaling of the \QTT\ inverse makes it
more efficient (relative to \tmatvec) as $N$ grows.
This is one of the features that allows us to compute the
\QTT\ for large $N$, and it implies that for sufficiently large problem
($N=\num{245760}$ in this example), inverse setup can become cheaper
than one matrix apply.

\para{Inverse storage}
\reftext[(b)]{fig:SlvMV-Y43-15-D4C0}{\figkeyword~} depicts the required
storage for each preconditioner as $\acc_p$ increases.
In this respect, \QTT\ is extremely efficient and memory requirements have
logarithmic scaling with respect to $N$ and do not exceed $100$~\abbrev{MB}.
On the other hand, the \HIF\ inverse displays $\O{N \log N}$ scaling
with large prefactor, requiring $10$s to $100$s of \abbrev{GB} of
memory to solve problems larger than a quarter million unknowns.

\para{Dependence on surface complexity} As we increase the surface
complexity by increasing the radial perturbation, we observe little to
no difference in iteration counts (\pr{tbl:TT-BIE-iters-d4c0}) as well
as the solve times for both \QTT\ and \HIF. Thus, both alleviate the
increase in condition number, evident by the increase in the
corresponding number of unpreconditioned iterations.
Though an increase in ranks and consequently in setup costs for the
\QTT\ inverse is observed, differences in rank distributions sharply
decrease with $\acc_p$, e.g., maximum ranks for $(\ell,m)=(8,7)$ case
are roughly $3\times$ higher for $\acc_p=\sci{-1}$ and $1.5\times$ for
$\acc_p=\sci{-3}$.
The increase in costs for the \HIF\ is predictably small due to the
fact that it focuses on compressing far range interactions.

\para{Effectiveness for different $\acc$ and $\acc_p$} To investigate
the robustness of the preconditioners for higher target accuracies, in
\pr{tbl:TT-BIE-targetacc}, we report how solve times vary across a
range of target and preconditioner accuracies.
Overall, the solve times (relative to \tmatvec) for both solvers seem to be
proportional to $\log(\acc)$, and their performance with respect to
$\acc_p$ seems to replicate the case observed in
\pr{fig:SlvMV-Y43-15-D4C0} (which corresponds to $\acc=\sci{-8}$).
This indicates that these direct solvers are extremely reliable as
preconditioners.
\begin{table}[!bt]
  \newcolumntype{M}{>{\centering\let\newline\\\arraybackslash\hspace{0pt}$}c<{$}}
  \newcolumntype{R}{>{\centering\let\newline\\\arraybackslash\hspace{0pt}$}r<{$}}
  \centering
  \small
  \begin{tabular}{MM RRRRR RRRRR}\toprule
    \multirow{2}{*}[-3pt]{$\acc$}
    &\multirow{2}{*}[-3pt]{\parbox[c]{.45in}{\centering Unprec.\\\Matvec{s}}}
    &\multicolumn{5}{c}{\QTT}
    &\multicolumn{5}{c}{\HIF}
    \\\cmidrule(lr){3-7} \cmidrule(lr){8-12}
               &      & \acc_p=\sci{-1} & \sci{-1.5} & \sci{-2} & \sci{-2.5} & \sci{-3} & \sci{-1} & \sci{-1.5} & \sci{-2} & \sci{-2.5} & \sci{-3} \\\midrule
     \sci{-4\0} & \081 & 13.2                    & 9.2       & 8.4     & 6.1       & 7.4     & 7.3     & 6.5       & 5.7     & 5.7       & 5.7     \\
     \sci{-6\0} & 107  & 19.4                    & 14.8      & 13.3    & 11.1      & 10.6    & 13.5    & 10.9      & 8.0     & 5.7       & 5.7     \\
     \sci{-8\0} & 141  & 25.5                    & 20.9      & 18.2    & 13.7      & 13.8    & 17.7    & 15.3      & 10.3    & 8.0       & 8.0     \\
     \sci{-10}  & 165  & 33.2                    & 25.4      & 23.0    & 17.2      & 17.0    & 22.4    & 16.6      & 12.7    & 10.3      & 10.3    \\
     \sci{-12}  & 187  & 39.0                    & 30.5      & 27.9    & 20.9      & 20.2    & 28.9    & 22.8      & 17.0    & 12.6      & 10.3    \\
     \bottomrule
  \end{tabular}
  \captionsetup{width=.95\linewidth}
  \mcaption{tbl:TT-BIE-targetacc}{Solve time for varying target and
    preconditioner accuracies}{
    Comparison of solve times relative to \tmatvec\
    for \QTT\ and \HIF\ preconditioners for target accuracies $\acc =
    \sci{-4} \text{ to } \sci{-12}$ and preconditioner accuracies
    $\acc_p =\sci{-1} \text{ to } \sci{-3}$ for the model surface with
    radius $\rho = 1 + 0.5Y_{4,3}$, and problem size $N=\num{245760}$
    ($n=512, p=15$).
    \Matvec\ counts for the unpreconditioned solve are also included for
    reference.}
\end{table}

\subsubsection{Perturbations of the benchmark}
To further quantify the effectiveness of the \QTT\ preconditioner, we
perform experiments in which we perturb the uniform cubic lattice
considered in the benchmark.
Given the ineffectiveness of the multigrid preconditioner presented
above, we only focus on the \HIF\ preconditioner for comparison.
Since one expects the \QTT\ decomposition to exploit regularities and
invariances in the geometry, this set of tests is aimed to measure its
robustness when the regularity of the lattice is broken.

We construct each surface in the lattice with a radius of the form
$\rho = 1 + r Y_{\ell m}$, where $r$ is a random number in $(0,1)$,
and $(\ell,m)$ is also randomly chosen between $(4,3)$ and
$(8,7)$.
Additionally, we perturb the location of the center of each surface in
a random direction by up to $50$ percent of its diameter.

The results of this test are shown in \pr{fig:SlvMV-Y43-15-d4c7}. We
observe similar behavior to the one seen in \pr{fig:SlvMV-Y43-15-D4C0}
for both solvers in terms of how solve times and setup costs behave as
functions of the number of surfaces and preconditioner accuracy
$\acc_p$. Comparing the corresponding data for the \QTT\ solver in these two
figures, we observe a moderate increase in inverse setup times and
storage for $n=256, \num{4096}\, (N \geq \num{245760})$.
This has an impact on the effective solve times in terms of {\matvec}s,
as the preconditioner apply is a bit more expensive. We also observe
that for both \QTT\ and \HIF, $\acc_p \simeq \sci{-1}$ seems to be
much less effective than in the benchmark case.
However, for $\acc_p \leq \sci{-2}$, iteration counts and solve times
display similar behavior to the benchmark case.

\begin{figure}[!tb]
  \centering
  \small
  \setlength\figureheight{.35\textwidth}
  \setlength\figurewidth{.5\textwidth}
  
  \beginpgfgraphicnamed{SlvMV_D4C7_15_multi}%
  \tikzsetnextfilename{tikz-external-SlvMV_D4C7_15_multi}%
  \scalebox{1}{\begin{tikzpicture}
  \begin{groupplot}[%
      ,group style={rows=1,columns=2,group name=cost-rand}
      ,width=\figurewidth
      ,height=\figureheight
      ,ymajorgrids
      ,xmajorgrids
      ,ytick scale label code/.code={v}
      ,xmode=log
      ,xminorticks=true
      ,ymin=0
      ,ymax=80
      ,legend style={legend columns=4,font=\footnotesize}
      ,legend cell align=left
    ]
    \nextgroupplot[%
      ,legend to name=rand-legend
      ,mark=none
      ,xmin=0.05
      ,xmax=50000
      ,xlabel={$\log_{10}\left({\text{Precond. Setup}}/{\tmatvec}\right)$}
      ,ylabel={Solve Time/\tmatvec}
    ]
    \addplot [color=plt-lc1,dash pattern=on 4pt off 2pt,line width=1.0pt,mark size=2.0pt,mark=o,mark options={solid}]
    table[row sep=crcr]{
      62.45      31.3064\\
      198.59     18.1657\\
      1265.70    15.3508\\
      2542.53    13.0975\\
      6171.52    11.6969\\
    };\addlegendentry{TT $(N=\num{3840})$}
    \addplot [color=plt-lc2,dash pattern=on 4pt off 2pt,line width=1.0pt,mark size=2.0pt,mark=triangle,mark options={solid}]
    table[row sep=crcr]{
      6.39      43.9279\\
      17.67     19.9255\\
      96.21     16.7710\\
      170.56    14.6624\\
      480.34    12.9338\\
    };\addlegendentry{TT $(N=\num{30720})$}
    \addplot [color=plt-lc3,dash pattern=on 4pt off 2pt,line width=1.0pt,mark size=2.5pt,mark=diamond,mark options={solid}]
    table[row sep=crcr]{
      0.77      67.5851\\
      2.84      21.4760\\
      10.63     19.4705\\
      28.27     17.4594\\
      89.21     17.2600\\
    };\addlegendentry{TT $(N=\num{245760})$}
    \addplot [color=plt-lc4,dash pattern=on 4pt off 2pt,line width=1.0pt,mark size=2.5pt,mark=x,mark options={solid}]
    table[row sep=crcr]{
      0.11      80.0000\\
      0.39      24.7098\\
      1.29      17.8842\\
      7.76      18.2962\\
      29.51     20.2796\\
    };\addlegendentry{TT $(N=\num{1996080})$}
    \addplot [color=plt-lc1,solid,line width=1.0pt,mark size=2.0pt,mark=o,mark options={solid}]
    table[row sep=crcr]{
      8.4301    15.282\\
      12.8051	  9.0699\\
      16.3133	  9.1995\\
      17.2379	  7.2680\\
      18.5764	  7.2521\\
    };\addlegendentry{HIF $(N=\num{3840})$}
    \addplot [color=plt-lc2,solid,line width=1.0pt,mark size=2.0pt,mark=triangle,mark options={solid}]
    table[row sep=crcr]{
      12.7600 15.3834\\
      15.4032 18.5055\\
      17.2126  9.3022\\
      18.3419  7.2096\\
      20.0557  7.2086\\
    };\addlegendentry{HIF $(N=\num{30720})$}
    \addplot [color=plt-lc3,solid,line width=1.0pt,mark size=2.5pt,mark=diamond,mark options={solid}]
    table[row sep=crcr]{
      77.1253 43.7022\\
      79.3136 47.0735\\
      80.2891 13.7421\\
      83.9392  7.3368\\
      87.7000  7.6033\\
    };\addlegendentry{HIF $(N=\num{245760})$}
    \nextgroupplot[%
      ,mark=none
      ,xmin=0.01
      ,xmax=10000
      ,xlabel={$\log_{10}$(Precond Memory) [MB]}
      ,yticklabels={\empty}]
    \addplot [color=plt-lc1,dash pattern=on 4pt off 2pt,line width=1.0pt,mark size=2.0pt,mark=o,mark options={solid},forget plot]
    table[row sep=crcr]{
      0.06     31.31\\
      1.07     18.17\\
      3.21     15.35\\
      4.88     13.10\\
      8.01     11.70\\
    };
    \addplot [color=plt-lc2,dash pattern=on 4pt off 2pt,line width=1.0pt,mark size=2.0pt,mark=triangle,mark options={solid},forget plot]
    table[row sep=crcr]{
      0.07     43.93\\
      0.94     19.93\\
      3.09     16.77\\
      4.95     14.66\\
      8.61     12.93\\
    };
    \addplot [color=plt-lc3,dash pattern=on 4pt off 2pt,line width=1.0pt,mark size=2.0pt,mark=diamond,mark options={solid},forget plot]
    table[row sep=crcr]{
      0.09     67.59\\
      1.24     21.48\\
      3.35     19.47\\
      5.14     17.46\\
      11.48     17.26\\
    };
    \addplot [color=plt-lc4,dash pattern=on 4pt off 2pt,line width=1.0pt,mark size=2.5pt,mark=x,mark options={solid},forget plot]
    table[row sep=crcr]{
      0.12     80.00\\
      1.90     24.71\\
      4.14     17.88\\
      12.21     18.30\\
      38.65     20.28\\
    };
    \addplot [color=plt-lc1,solid,line width=1.0pt,mark size=2.0pt,mark=o,mark options={solid},forget plot]
    table[row sep=crcr]{
      19.26     15.28\\
      38.31      9.07\\
      48.39      9.20\\
      54.58      7.27\\
      59.77      7.25\\
    };
    \addplot [color=plt-lc2,solid,line width=1.0pt,mark size=2.0pt,mark=triangle,mark options={solid},forget plot]
    table[row sep=crcr]{
      136.22     15.38\\
      302.35     18.51\\
      415.46      9.30\\
      487.06      7.21\\
      573.06      7.21\\
    };
    \addplot [color=plt-lc3,solid,line width=1.0pt,mark size=2.0pt,mark=diamond,mark options={solid},forget plot]
    table[row sep=crcr]{
      1130.57     43.70\\
      2558.85     47.07\\
      3786.83     13.74\\
      5002.30      7.34\\
      6567.29      7.60\\
    };
  \end{groupplot}
  \node at ($(cost-rand c1r1.north)!.5!(cost-rand c2r1.north)$) [anchor=south] {\ref{rand-legend}};
  \node at (cost-rand c1r1.north) [anchor=south, yshift=-.11\figureheight, xshift=.35\figurewidth]{(a)};
  \node at (cost-rand c2r1.north) [anchor=south, yshift=-.11\figureheight, xshift=.35\figurewidth]{(b)};
\end{tikzpicture}%
}%
  \endpgfgraphicnamed%
\\
  \mcaption{fig:SlvMV-Y43-15-d4c7}{Solve vs. preconditioner setup and
    memory costs for irregular lattice and shapes}{Semi-log plots for
    Solve times, normalized by \matvec\ time \tmatvec, for a given
    inversion time, also normalized by \tmatvec, and storage
    requirements (in \abbrev{MB}) for \TT\ and \HIF\ preconditioners.
    Each model surface has radius $\rho = 1 + r Y_{\ell,\ell-1}$,
    where $\ell$ is randomly chosen to be 4 or 8 and $r$ is a
    random number in $(0,1)$.
    The location of the surfaces is also randomly perturbed.
    }
\end{figure}

\para{Distance between Surfaces}
As mentioned in \pr{sec:intro}, it is
well known that conditioning of second kind integral equations tends
to deteriorate as surfaces come close to contact. In order to test
robustness in performance of the \QTT\ and \HIF\ preconditioners, we compare speedups and setup costs as we
draw surface centers in the lattice close to each other in our experiments.

In \pr{fig:SpUp-vs-logdst}, we report the speedups with respect to the
average unpreconditioned solve, and plot them against the logarithm of
the distance between surfaces in the lattice.
Here, we use the number of unpreconditioned iterations as a surrogate
for the problem conditioner number.
As we draw the surfaces together, we observe that iteration counts and
solve times for both preconditioners show a slight increase for low
accuracy ($\acc_p \simeq \sci{-1}$), becoming almost independent of
distance for higher preconditioner accuracies ($\acc_p \leq
\sci{-2}$).
This causes the effective speedup to increase as the unpreconditioned
solve becomes more expensive, due to the increase in condition number.
\begin{figure}[!tb]
  \centering
  \small
  \setlength\figureheight{1.3in}
  \setlength\figurewidth{2.5in}
  
  \beginpgfgraphicnamed{SpUp_vs_logdst}%
  \tikzsetnextfilename{tikz-external-SpUp_vs_logdst}%
  \scalebox{1}{\begin{tikzpicture}
  \begin{axis}[%
      width=\figurewidth,
      height=\figureheight,
      scale only axis,
      xmode=log,
      xmin=0.4,
      xmax=100,
      ymajorgrids,
      xmajorgrids,
      xminorticks=true,
      xlabel={$\displaystyle -\log_{10}\left(\text{Distance}\right)$},
      ymin=0,
      ymax=100,
      ytick={ 0, 20, 40, 60, 80, 100},
      ylabel={Speedup},
      legend cell align=left,
      legend pos=outer north east,
      legend style={font=\footnotesize},
    ]
    \addplot [color=plt-lc1,solid,line width=1.0pt,mark size=2.0pt,mark=o,mark options={solid}]
    table[row sep=crcr]{0.549540873857625	27.6\\
      1.25892541179417	33.73\\
      3.31131121482591	66.09\\
      19.9526231496888	85.06\\
      50.1187233627272	87.05\\
    };\addlegendentry{HIF $(N=\num{3840})$}
    \addplot [color=plt-lc2,solid,line width=1.0pt,mark size=2.0pt,mark=triangle,mark options={solid}]
    table[row sep=crcr]{0.549540873857625	13.77\\
      1.25892541179417	17.75\\
      3.31131121482591	26.96\\
      19.9526231496888	25.29\\
      50.1187233627272	57.28\\
    };\addlegendentry{HIF $(N=\num{30720})$}
    \addplot [color=plt-lc3,solid,line width=1.0pt,mark size=2.5pt,mark=diamond,mark options={solid}]
    table[row sep=crcr]{0.549540873857625	15.35\\
      1.25892541179417	15.22\\
      3.31131121482591	12.14\\
      19.9526231496888	16.48\\
      50.1187233627272	35.14\\
    };\addlegendentry{HIF $(N=\num{245760})$}
    \addplot [color=plt-lc1,dash pattern=on 4pt off 2pt,line width=1.0pt,mark size=2.0pt,mark=o,mark options={solid}]
    table[row sep=crcr]{0.549540873857625	13.82\\
      1.25892541179417	13.54\\
      3.31131121482591	39.56\\
      19.9526231496888	40.22\\
      50.1187233627272	40.99\\
    };\addlegendentry{TT $(N=\num{3840})$}
    \addplot [color=plt-lc2,dash pattern=on 4pt off 2pt,line width=1.0pt,mark size=2.0pt,mark=triangle,mark options={solid}]
    table[row sep=crcr]{0.549540873857625	9.87\\
      1.25892541179417	7.09\\
      3.31131121482591	18.04\\
      19.9526231496888	14.81\\
      50.1187233627272	25.45\\
    };\addlegendentry{TT $(N=\num{30720})$}
    \addplot [color=plt-lc3,dash pattern=on 4pt off 2pt,line width=1.0pt,mark size=2.5pt,mark=diamond,mark options={solid}]
    table[row sep=crcr]{0.549540873857625	11.65\\
      1.25892541179417	6.37\\
      3.31131121482591	8.75\\
      19.9526231496888	10.36\\
      50.1187233627272	26.29\\
    };\addlegendentry{TT $(N=\num{245760})$}
    \addplot [color=plt-lc4,dash pattern=on 4pt off 2pt,line width=1.0pt,mark size=2.5pt,mark=x,mark options={solid}]
    table[row sep=crcr]{0.549540873857625	8.64\\
      1.25892541179417	5.75\\
      3.31131121482591	9.97\\
      19.9526231496888	13.5\\
      50.1187233627272	22.74\\
    };\addlegendentry{TT $(N=\num{1996080})$}
  \end{axis}
\end{tikzpicture}%
}%
  \endpgfgraphicnamed%
\\
  \captionsetup{width=.85\linewidth}
  \mcaption{fig:SpUp-vs-logdst}{Solve speedup vs. log distance between
    surfaces}{Speedups, defined as the ratio between unpreconditioned
    and preconditioned solves and excluding the preconditioner setup time
    (see \pr{fig:StpandMem-vs-logdst} for the setup costs), vs.
    log of the surface spacing in the lattice for \QTT\ and \HIF\ preconditioners with $\acc_p = \sci{-3}$.}
\end{figure}

\begin{figure}[!tb]
  \centering
  \small
  \setlength\figureheight{.36\textwidth}
  \setlength\figurewidth{.5\textwidth}
  
  \beginpgfgraphicnamed{StpandMem_vs_logdst}%
  \tikzsetnextfilename{tikz-external-StpandMem_vs_logdst}%
  \scalebox{1}{\begin{tikzpicture}
  \begin{groupplot}[%
      ,group style={rows=1,columns=2,group name=setupmem,horizontal sep=4em}
      ,width=\figurewidth
      ,height=\figureheight
      ,ymajorgrids
      ,xmajorgrids
      ,ytick scale label code/.code={v}
      ,xmode=log
      ,ymode=log
      ,xminorticks=true
      ,yminorticks=true
      ,legend style={legend columns=4,font=\footnotesize}
      ,legend cell align=left
      ,ylabel style={align=center}
    ]
    \nextgroupplot[%
      ,legend to name=setupmem-legend
      ,mark=none
      ,xmin=0.4
      ,xmax=100
      ,ymin=0.005
      ,ymax=100
      ,ylabel={$\log_{10}$(Setup) [Unprec. solves]}
      ,xlabel={$\displaystyle -\log_{10}\left(\text{Distance}\right)$}
    ]
    \addplot [color=plt-lc1,dash pattern=on 1pt off 3pt on 3pt off 3pt,line width=1.0pt,mark size=2.0pt,mark=o,mark options={solid}]
    table[row sep=crcr]{0.549540873857625	6.6\\
      1.25892541179417	8.51\\
      3.31131121482591	17.71\\
      19.9526231496888	19.75\\
      50.1187233627272	42.93\\
    };\addlegendentry{TT $(N=\num{3840})$}
    \addplot [color=plt-lc2,dash pattern=on 1pt off 3pt on 3pt off 3pt,line width=1.0pt,mark size=2.0pt,mark=triangle,mark options={solid}]
    table[row sep=crcr]{0.549540873857625	0.94\\
      1.25892541179417	1.5\\
      3.31131121482591	9.79\\
      19.9526231496888	12.08\\
      50.1187233627272	6.84\\
    };\addlegendentry{TT $(N=\num{30720})$}
    \addplot [color=plt-lc3,dash pattern=on 1pt off 3pt on 3pt off 3pt,line width=1.0pt,mark size=2.5pt,mark=diamond,mark options={solid}]
    table[row sep=crcr]{0.549540873857625	0.1\\
      1.25892541179417	0.2\\
      3.31131121482591	1.93\\
      19.9526231496888	1.37\\
      50.1187233627272	0.94\\
    };\addlegendentry{TT $(N=\num{245760})$}
    \addplot [color=plt-lc4,dash pattern=on 1pt off 3pt on 3pt off 3pt,line width=1.0pt,mark size=2.5pt,mark=x,mark options={solid}]
    table[row sep=crcr]{0.549540873857625	0.01\\
      1.25892541179417	0.02\\
      3.31131121482591	0.24\\
      19.9526231496888	0.23\\
      50.1187233627272	0.11\\
    };\addlegendentry{TT $(N=\num{1996080})$}
    \addplot [color=plt-lc1,solid,line width=1.0pt,mark size=2.0pt,mark=o,mark options={solid}]
    table[row sep=crcr]{0.549540873857625	0.13\\
      1.25892541179417	0.06\\
      3.31131121482591	0.04\\
      19.9526231496888	0.03\\
      50.1187233627272	0.03\\
    };\addlegendentry{HIF $(N=\num{3840})$}
    \addplot [color=plt-lc2,solid,line width=1.0pt,mark size=2.0pt,mark=triangle,mark options={solid}]
    table[row sep=crcr]{0.549540873857625	0.24\\
      1.25892541179417	0.15\\
      3.31131121482591	0.26\\
      19.9526231496888	0.51\\
      50.1187233627272	0.24\\
    };\addlegendentry{HIF $(N=\num{30720})$}
    \addplot [color=plt-lc3,solid,line width=1.0pt,mark size=2.5pt,mark=diamond,mark options={solid}]
    table[row sep=crcr]{0.549540873857625	0.88\\
      1.25892541179417	0.82\\
      3.31131121482591	1.57\\
      19.9526231496888	2.11\\
      50.1187233627272	1.07\\
    };\addlegendentry{HIF $(N=\num{245760})$}
    \nextgroupplot[%
      ,mark=none
      ,xmin=0.4
      ,xmax=100
      ,ymin=1
      ,ymax=100000
      ,ylabel={$\log_{10}$(Storage) [MB]}
      ,xlabel={$\displaystyle -\log_{10}\left(\text{Distance}\right)$}
    ]
    \addplot [color=plt-lc1,dash pattern=on 1pt off 3pt on 3pt off 3pt,line width=1.0pt,mark size=2.0pt,mark=o,mark options={solid}]
    table[row sep=crcr]{0.549540873857625	2.54\\
      1.25892541179417	4.17\\
      3.31131121482591	7.99\\
      19.9526231496888	9.33\\
      50.1187233627272	13.24\\
    };
    \addplot [color=plt-lc2,dash pattern=on 1pt off 3pt on 3pt off 3pt,line width=1.0pt,mark size=2.0pt,mark=triangle,mark options={solid}]
    table[row sep=crcr]{0.549540873857625	3.14\\
      1.25892541179417	5.17\\
      3.31131121482591	10.1\\
      19.9526231496888	16.59\\
      50.1187233627272	17.02\\
    };
    \addplot [color=plt-lc3,dash pattern=on 1pt off 3pt on 3pt off 3pt,line width=1.0pt,mark size=2.5pt,mark=diamond,mark options={solid}]
    table[row sep=crcr]{0.549540873857625	3.6\\
      1.25892541179417	6.06\\
      3.31131121482591	12.05\\
      19.9526231496888	19.95\\
      50.1187233627272	20.46\\
    };
    \addplot [color=plt-lc4,dash pattern=on 1pt off 3pt on 3pt off 3pt,line width=1.0pt,mark size=2.5pt,mark=x,mark options={solid}]
    table[row sep=crcr]{0.549540873857625	4.03\\
      1.25892541179417	6.81\\
      3.31131121482591	14.37\\
      19.9526231496888	22.51\\
      50.1187233627272	23.25\\
    };
    \addplot [color=plt-lc1,solid,line width=1.0pt,mark size=2.0pt,mark=o,mark options={solid}]
    table[row sep=crcr]{0.549540873857625	58\\
      1.25892541179417	70\\
      3.31131121482591	104\\
      19.9526231496888	150\\
      50.1187233627272	160\\
    };
    \addplot [color=plt-lc2,solid,line width=1.0pt,mark size=2.5pt,mark=triangle,mark options={solid}]
    table[row sep=crcr]{0.549540873857625	557\\
      1.25892541179417	880\\
      3.31131121482591	1830\\
      19.9526231496888	3064\\
      50.1187233627272	3247\\
    };
    \addplot [color=plt-lc3,solid,line width=1.0pt,mark size=2.5pt,mark=diamond,mark options={solid}]
    table[row sep=crcr]{0.549540873857625	6237\\
      1.25892541179417	13237\\
      3.31131121482591	27124\\
      19.9526231496888	45417\\
      50.1187233627272	48131\\
    };
  \end{groupplot}
  \node at ($(setupmem c1r1.north)!.5!(setupmem c2r1.north)$) [anchor=south,yshift=2pt] {\ref{setupmem-legend}};
  \node at (setupmem c1r1.north) [anchor=south, yshift=-.11\figureheight, xshift=.35\figurewidth]{(a)};
  \node at (setupmem c2r1.north) [anchor=south, yshift=-.11\figureheight, xshift=.35\figurewidth]{(b)};
\end{tikzpicture}%
}%
  \endpgfgraphicnamed%
\\
  \mcaption{fig:StpandMem-vs-logdst}{Preconditioner setup and memory
    vs. log distance}{Log of setup costs (in unpreconditioned solves)
    and storage requirements vs. log of the surface spacing in the
    lattice for \QTT\ and \HIF\ preconditioners with $\acc_p =
    \sci{-3}$.}
\end{figure}

In \pr{fig:StpandMem-vs-logdst}, we plot setup costs against the log of
the distance between surfaces.
Although displaying a sharp increase at first, preconditioner setup
times increase at a pace much slower than the cost of the
unpreconditioned solve, becoming more efficient.
We again observe that as $N$ increases, the \QTT\ preconditioner
becomes more cost-effective, becoming just a fraction of an
unpreconditioned solve for $N=\num{1966080}$.
The rate at which storage requirements increase for both solvers also
slows down as we reduce the distance, which means that they both
display robust behavior in spite of the added complexity.

\section{Conclusions\label{sec:conclusions}}
Motivated by the ongoing challenges to produce memory-efficient and
reliable fast solvers for integral equations in complex geometries, we
presented a robust framework employing the Tensor Train decomposition
to accelerate the solution of volume and boundary integral equations
in three dimensions.

In the context of volume integral equations on a regularly sampled
domain, even for problems with up to 10 million unknowns and for
relatively high target accuracies ($\acc=\sci{-6}$ to $\sci{-10}$), we
are able to produce a compressed inverse in no more than a few minutes
and store it using tens of \abbrev{MB}s of memory, When compared to the current
state-of-the-art direct solvers in three dimensions, \QTT\ is the only
such fast direct solver to retain practical performance for large
problem sizes.

In \pr{ssc:bie-lattice}, we showed that the \QTT\ framework is
applicable to matrices arising from the discretization of boundary
integral equations in complex, multiply-connected geometries.
Nonetheless, for a given target accuracy, the \QTT\ ranks are
considerably higher than the volume integral case, rendering high
target accuracies impractical.
We thus proposed using a low target accuracy (e.g., $\acc_p$ between
$\sci{-1}$ and $\sci{-3}$) version of the \QTT-based inverse as a
preconditioner for the \GMRES.
We compared its effectiveness and cost against two alternative
preconditioners, a simple multigrid and a low-accuracy \HIF\ inverse.

By virtue of being a multilevel, low-accuracy direct solver, the
\QTT\ preconditioner matches the \HIF\ in terms of reliability and
robustness across all examples presented.
We observe a clear trade-off between these two solvers: while \HIF\
generally provides higher speedups, modest setup costs and storage
requirements for the \QTT\ make it extremely cost-effective. This is
particularly the case for problems with a large number of unknowns, as
setting up the \QTT\ preconditioner typically becomes comparable to one
solve for the range of target and preconditioner accuracies presented.
In \pr{ssc:bie-lattice}, this allowed us to solve an external
Dirichlet problem for the Laplace equation in an irregular domain
composed of $4096$ surfaces. The setup time for the
\QTT\ preconditioner for this example is as small as a few
\FMM\ applies and has a memory footprint of less than
$10$~\abbrev{MB}.


There are several extensions of this work that can be pursued.
Here, all presented examples used a uniform hierarchical partition
of the domain, corresponding to a uniform tree.
We believe it is possible to extend the current implementation of the
\QTT\ algorithms to adaptive decompositions, using the connections
between Tensor Train and other hierarchical decomposition
techniques.
Another research direction is generalizing of the \QTT\ algorithms to
obtain effective parallel scaling.
It would be interesting to see whether the compute-bound nature of the
\QTT\ algorithms could be exploited to obtain good practical scaling.


\section{Acknowledgments}
We extend our thanks to Sergey Dolgov, Leslie Greengard, Ivan
Oseledets, and Mark Tygert for stimulating conversations about various
aspects of this work.
We are also grateful to Matthew Morse, Michael O'Neil, and Shravan
Veerapaneni for critical reading of this work and providing valuable
feedback.
We thank Kenneth~Ho and Lexing~Ying for kindly supplying the
\HIF\ code.
Support for this work was provided by the US National Science
Foundation (\abbrev{NSF}) under grant DMS-1320621.
E.C. acknowledges the support of the US National Science
Foundation (\abbrev{NSF}) through grant DMS-1418964.

\appendix
\section{\QTT\ decomposition as a hierarchical linear filter \label{apx:appendix_TT-filter}}
In \cite{oseledets2011algebraic}, the \QTT\ decomposition is
interpreted as a subspace approach, and in the case of the \abbrev{SVD} based
\QTT\ decomposition (or more generally, when tensor cores are
orthogonalized), as a fast method to compute a reduced orthogonal
basis for structured tensors.
We first show that the \QTT\ decomposition can be viewed as a linear
filter with respect to each of its cores.
We recall from \pr{sec:background} that the compression algorithm
proceeds by computing a sequence of low rank decompositions of
unfolding matrices.
For the \ordinal{(k-1)} unfolding matrix, we have the low rank
decomposition
\begin{equation}
\laT{A}_{\tree{T}}^{k-1}(\overline{i_1\dots i_{k-1}},
\overline{i_k\dots i_d}) = \laT{U}_{k-1}(\overline{i_1\dots i_{k-1}} ,
\alpha_{k-1}) \laT{V}_{k-1}(\alpha_{k-1} , \overline{i_k\dots i_d}).
\end{equation}

We then vectorize $\tensor{A}_{\tree{T}}$, using the formula for
products of matrices $\vec(\laT{ABC}) = (\laT{C}^{T} \otimes
\laT{A})\vec(\laT{B})$, to obtain
\begin{equation}
  \vec(\tensor{A}_{\tree{T}}) = (\laT{I}_{m_{k-1}} \otimes \laT{U}_{k-1})\vec(\laT{V}_{k-1}),
  \label{eq:filter_left}
\end{equation}
where $\laT{I}_{m_{k-1}}$ is the identity of size $m_{k-1} = \prod_{q
  \geq k} {n_q}$, and so the left factor is in fact a block-diagonal
matrix of $m_{k-1}$ copies of $\laT{U}_{k-1}$.
If $\laT{U}_{k-1}$ is orthogonal this factor is orthogonal as well,
and we also have that $\laT{V}_{k-1} =
\laT{U}_{k-1}^{\star}\laT{A}^{k-1}_{\tree{T}}$ and $\vec(\laT{V}_{k-1}) =
(\laT{I}_{m_{k-1}} \otimes \laT{U}_{k-1}^{\star})\vec(\tensor{A}_{\tree{T}})$.
Applying this same identity, we obtain
\begin{equation}
  \vec(\tensor{A}_{\tree{T}}) = (\laT{V}^{T}_{k-1} \otimes \laT{I}_{p_{k-1}})\vec(\laT{U}_{k-1}),
  \label{eq:filter_right}
\end{equation}
where $p_{k-1} = \prod_{q < k} {n_q}$, and again if
$\laT{V}^{T}_{k-1}$ is orthogonal, $\laT{U}_{k-1} =
(\laT{V}^{T}_{k-1})^{\star}\laT{A}^{k-1}_{\tree{T}}$ and
$\vec(\laT{U}_{k-1}) = (\laT{V}^{T}_{k-1} \otimes
\laT{I}_{p_{k-1}})^{\star}\vec(\tensor{A}_{\tree{T}})$.

Since the \QTT\ compression algorithm can proceed via sweeps of low
rank decompositions of unfoldings for $\laT{V}_{k}$ (left to right) or
of $\laT{U}_{k}$ (right to left), we can iterate
\pr{eq:filter_left} and \pr{eq:filter_right} separating one core at a
time.
Let $\tensor{G}_k$ denote the \ordinal{$k$} core of the
\QTT\ decomposition of $\laT{A}$, and $\laT{G}_k = \reshape
(\tensor{G}_k, [r_{k-1}n_k,r_k])$.
If we apply \pr{eq:filter_left} $k-1$ times, we obtain
\begin{equation}
  \vec(\tensor{A}_{\tree{T}}) = (\laT{I}_{m_1} \otimes
  \laT{G}_{1})(\laT{I}_{m_2} \otimes \laT{G}_{2}) \dots
  (\laT{I}_{m_{k-1}} \otimes \laT{G}_{k-1})\vec(\laT{V}_{k-1}).
\end{equation}
If we apply \pr{eq:filter_right}, we can write an explicit formula for
the \ordinal{$k$} core $\tensor{G}_k$
\begin{equation}
  \vec(\tensor{A}_{\tree{T}}) = (\laT{I}_{m_1} \otimes \laT{G}_{1}) \dots
  (\laT{I}_{m_{k-1}} \otimes \laT{G}_{k-1}) (\laT{G}^{T}_{d} \otimes
  \laT{I}_{p_{d}r_{k-1}}) \dots (\laT{G}^{T}_{k+1} \otimes
  \laT{I}_{p_{k+1}r_{k-1}}) \vec(\tensor{G}_{k})
  \label{eq:TT_reducedbasis}
\end{equation}

Following notation from \cite{dolgov2013amen1}, we denote the product
of the $d-1$ matrix factors in \pr{eq:TT_reducedbasis} as
$\tensor{P}_{\neq k}(\laT{A})$.
Its columns form a reduced tensor basis generated by $\{ \laT{G}_q
\}_{q\neq k}$, and $\vec(\tensor{G}_k)$ corresponds to the
coefficients that reconstruct $\laT{A}$.
If a correction step (e.g., using \abbrev{QR}) is implemented to make
$\{\laT{G}_{q}\}_{q<k}$ and $\{\laT{G}^{T}_{q}\}_{q>k}$ orthogonal,
then $\tensor{P}_{\neq k}(\laT{A})$ is orthogonal as well, and we have
\begin{equation}
  \vec(\tensor{G}_{k}) = (\laT{G}^{T}_{k+1} \otimes
  \laT{I}_{p_{k+1}r_{k-1}})^\star \dots (\laT{I}_{m_1} \otimes
  \laT{G}_{1})^\star \vec(\tensor{A}_{\tree{T}}) = \tensor{P}_{\neq k}(\laT{A})^\star
  \vec(\tensor{A}_{\tree{T}})
\label{eq:TT_applyfilter}
\end{equation}

Finally, we note that from \pr{eq:TT_applyfilter} an identity for
$\vec(\laT{A})$ can be readily found, as there exists a permutation
matrix $\Pi_{\tree{T}}$ such that $\vec(\laT{A}) = \Pi_{\tree{T}}
\vec(\tensor{A}_{\tree{T}})$.

\section{Setup of local linear systems for \QTT\ inversion algorithms \label{apx:appendix_inversion}}
The approximate inversion seeks to find $\la{X}$ that satisfies
$\laT{AX = I_N}$ or $\laT{AX + XA = 2I_N}$ and to extract a
\QTT\ decomposition for $\laT{X}$.
Given a set of proposed cores $\{ \tensor{W}_k \}_{k=1}^{d}$ for
$\laT{X}$ and fixing all but $\tensor{W}_{k}$, \pr{eq:TT_reducedbasis}
provides an explicit formula to interpret the \QTT\ decomposition as an
expansion in an orthonormal basis, with elements that depend on
$\{\tensor{W}_{j}\}_{j \neq k}$ and coefficients
$\vec(\tensor{W}_{k})$.
That is, we can write $\vec(\laT{X}) = \tensor{P}_{\neq k}(\laT{X})
\vec(\tensor{W}_{k})$ where $\tensor{P}_{\neq k}(\laT{X})$ is orthogonal.
\pr{eq:inverse} then becomes
\begin{equation}
  (\laT{I}_N \otimes \laT{A}) \tensor{P}_{\neq
    k}(\laT{X})\vec(\tensor{W}_{k}) = \vec(\laT{I}_N).
  \label{eq:system_Wk}
\end{equation}

From this one may readily notice that fixing all cores but
$\tensor{W}_k$ yields an overdetermined linear system.
Applying $\tensor{P}_{\neq k}(\laT{X})^\star$, we obtain the
equivalent reduced system
\begin{equation}
  \tensor{P}_{\neq k}(\laT{X})^\star(\laT{I}_N \otimes \laT{A}) \tensor{P}_{\neq
    k}(\laT{X}) \vec(\tensor{W}_{k}) = \tensor{P}_{\neq
    k}(\laT{X})^\star\vec(\laT{I}_N).
  \label{eq:redsystem_Wk}
\end{equation}
The matrix in the linear system above is of size $n_{k}r_{k-1}r_k
\times n_{k}r_{k}r_{k-1}$.  \pr{eq:redsystem_Wk} allows us to solve
for each core $\tensor{W}_k$ by projecting $\laT{X}$ onto this reduced
basis in which $\tensor{W}_k$ is the only degree of freedom.
We note that this necessarily leaves the size and rank of
$\tensor{W}_k$ fixed, and as a consequence, some strategy must be
implemented to increase ranks and accelerate convergence towards the
solution.
\begin{itemize}
\item In \abbrev{DMRG} methods, this issue is resolved by contracting two cores
  $\tensor{W}_k,\tensor{W}_{k+1}$ at a time into a supercore
  \begin{equation} \tensor{S}_{k}(\alpha_{k-1},\overline{i_k
      i_{k+1}},\alpha_{k+1}) = \tensor{W}_k (\alpha_{k-1},i_k,\alpha_k)
    \tensor{W}_{k+1}(\alpha_k,i_{k+1},\alpha_{k+1}), \end{equation}
  and solving the corresponding reduced system for $\tensor{S}_k$.
  The \ordinal{k} rank is now free (it was contracted in merging both
  cores) and it will be determined when newly computed $\tensor{S}_k$
  is split into cores.
\item In \abbrev{AMEN}-type methods \cite{dolgov2013amen1,dolgov2013amen2},
  the residual $\la{R}$ of this system is approximated in \QTT\ form, and
  then it is used in an enrichment step to expand the reduced basis
  and allow for ranks to increase.
\end{itemize}

We also note that even though we present these reduced systems
explicitly, in both \abbrev{DMRG} and \abbrev{AMEN} methods a recursive formula is
employed to construct them as they cycle through the cores of
$\laT{X}$.
In fact, it is shown in \cite{oseledets2012solution} that if
$\tensor{G}_k (\alpha_{k-1}, \overline{i_kj_k},\alpha_k)$ is the
\ordinal{k} core of the \QTT\ decomposition of $(\laT{I}_N \otimes
\laT{A})$, then we can write the left-hand side of
\pr{eq:redsystem_Wk} in the following form \begin{equation} \sum_{\alpha,\gamma,j}
  {\Psi_{k-1}(\alpha_{k-1},\beta_{k-1},\gamma_{k-1})
    \tensor{G}_k(\alpha_{k-1},\overline{i_kj_k},\alpha_k)
    \Phi_{k}(\alpha_k,\beta_{k},\gamma_{k})
    \tensor{W}_k(\gamma_{k-1},\overline{j_kp_k},\gamma_k)}
  \label{eq:inverse_Wk_TT},
\end{equation} where $\Psi_{k-1}$ is a function of $\{ \tensor{G}_j
\}_{j=1}^{k-1}$ and $\{ \tensor{W}_j \}_{j=1}^{k-1}$, and $\Phi_k$ a
function of $\{ \tensor{G}_j \}_{j=k+1}^{d}$ and $\{ \tensor{W}_j
\}_{j=k+1}^{d}$.
Moreover, recursive formulas $\psi$ and $\phi$ exist such that
\begin{equation}
  \Psi_k =
  \psi(\Psi_{k-1},\tensor{G}_k,\tensor{W}_k)\quad\text{and}\quad
  \Phi_{k+1} = \phi(\Phi_{k-1},\tensor{G}_k,\tensor{W}_k).
\end{equation}
Finally, we note that \pr{eq:inverse_Wk_TT} may be interpreted as a
compressed form of the matrix in the reduced linear system with a
3-dimensional \QTT\ decomposition with cores $\Psi$, $\tensor{G}$, and
$\Phi$.  If this linear system is relatively small, the matrix in
\pr{eq:redsystem_Wk} may be formed to solve this system
densely. Otherwise, it is preferable to use the fast \QTT\ matrix
vector apply on the tensor decomposition in \pr{eq:inverse_Wk_TT} to
solve this system using an iterative algorithm such as \GMRES\ or
\abbrev{Bi-CGSTAB}.

\printbibliography

\end{document}